\documentclass[reqno]{amsart}

\usepackage{graphicx} 
\usepackage{amssymb,amsfonts,amsmath,amsthm,amscd} 
\usepackage{multirow}

\usepackage{mathpazo}
\usepackage{palatino}
\usepackage{dsfont}
\usepackage[scaled]{helvet} 
\usepackage{courier} 
\usepackage{mathrsfs} 
\normalfont

\usepackage{mathtools} 
\usepackage[all]{xy}
\usepackage[T1]{fontenc}
\usepackage{calligra}
\usepackage{verbatim}
\usepackage[usenames,dvipsnames]{color}
\usepackage[utf8x]{inputenc}
\usepackage{hyperref}
\usepackage{cite}
\usepackage{enumitem}
\usepackage{subcaption}
\usepackage[font=small,labelfont=bf]{caption}
\theoremstyle{plain}
\newtheorem{theorem}[equation]{Theorem}

\newtheorem{corollary}[equation]{Corollary}
\newtheorem{lemma}[equation]{Lemma}
\newtheorem{proposition}[equation]{Proposition}

\theoremstyle{definition}
\newtheorem{example}[equation]{Example}
\newtheorem{definition}[equation]{Definition}
\newtheorem{problem}[equation]{Problem}
\newtheorem{remark}[equation]{Remark}

\numberwithin{equation}{section} 
\newcommand{\A}{\mathcal{A}}

\newcommand{\m}{\vec{m}}
\renewcommand{\P}{\vec{P}}
\renewcommand{\vec}[1]{\mathbf{#1}}
\newcommand{\semigroup}[1]{\langle #1 \rangle}

\newcommand{\eng}{\mathcal{E}(\mathbf{n},\mathbf{g})}
\newcommand{\shom}{\underline{\operatorname{hom}}}
\newcommand{\be}{{\mathbf e}}

\newcommand{\N}{{\mathds{N}}}
\newcommand{\Z}{{\mathds{Z}}}
\newcommand{\Q}{{\mathds{Q}}}
\newcommand{\R}{{\mathds{R}}}
\renewcommand{\P}{{\mathds{P}}}
\newcommand{\C}{{\mathds{C}}}
\newcommand{\ZZ}{\mathsf{Z}}
\newcommand{\B}{{\mathcal{B}}}
\newcommand{\M}{{\mathcal{M}}}

\newcommand{\g}{\gamma}
\newcommand{\E}{\mathcal{E}}
\newcommand{\bg}{\mathbf{g}}
\newcommand{\n}{\mathbf{n}}

\renewcommand{\geq}{\geqslant}
\renewcommand{\leq}{\leqslant}
\renewcommand{\phi}{\varphi}

\usepackage{enumitem}
\setenumerate{itemsep=5pt} 
\setitemize{itemsep=5pt, leftmargin=10pt} 

\allowdisplaybreaks

\graphicspath{{figures/}}

\newcommand{\twovector}[2]{\big[ \begin{smallmatrix} #1 \\ #2 \end{smallmatrix} \big]}

\usepackage{supporting_files/tikzcd}
\tikzcdset{row sep=large}

\title[Rapidly growing AF algebras]{Rapidly growing AF algebras}
\author[K.~Aguilar]{Konrad Aguilar}
 \address{Department of Mathematics \\
        Pomona College \\
        Claremont \\
        CA 91711 \\
         U.S.A}
 \email{konrad.aguilar@pomona.edu}

 \author[S.R.~Garcia]{Stephan Ramon Garcia}
 \address{Department of Mathematics \\
        Pomona College \\
        Claremont \\
        CA 91711 \\
         U.S.A}
\urladdr{\url{https://stephangarcia.sites.pomona.edu/}}
 \email{stephan.garcia@pomona.edu}
 
\author[E.~Knight]{Evelyne Knight}
 \address{Department of Mathematics \\
        The University of Chicago \\
        Chicago \\
        IL 60637 \\
         U.S.A}
 \email{evelynek@uchicago.edu}

 \author[C.~Marple]{Chloe Marple}
 \address{Department of Mathematics \\
        Pomona College \\
        Claremont \\
        CA 91711 \\
         U.S.A}
 \email{ckme2022@mymail.pomona.edu}

 \author[J. Spielberg]{Jack Spielberg}
  \address{School of Mathematical and Statistical Sciences \\
        Arizona State University \\
        901  S Palm Walk\\
       Tempe, AZ 85281 \\
         U.S.A}
 \email{spielberg@asu.edu}


\begin{document}
    \begin{abstract}
        We introduce certain families of AF algebras associated to Bratteli diagrams arising from numerical semigroup theory, a branch of combinatorics.  Curry--Schoenberg B-splines, staples of computer-aided design, provide insight into the statistical properties of these algebras.  This permits us to consider certain ensembles of ``rapidly growing'' AF algebras from a probabilistic viewpoint.  
    \end{abstract}

\maketitle


\section{Introduction}\label{Section:Introduction}

Our aim in this article is to initiate the study of rapidly growing C*-algebras, 
a novel class of approximately finite-dimensional C*-algebras that can be simultaneously be regarded as generalizations of uniformly hyperfinite algebras and as natural bridges between functional analysis and combinatorics.

An examination of the foundations of AF-algebra theory suggests that a connection to combinatorics should not be unexpected. Recall that a $C^*$-algebra (or \emph{algebra} for short) is \emph{approximately finite dimensional} (AF) if it is the closure of an increasing union of finite-dimensional algebras.  Since any finite-dimensional C*-algebra is a direct sum of full matrix algebras \cite[Thm.~III.1.1]{davidson} and because unital *-homomorphisms between such algebras are determined, up to unitary equivalence, by associated linear diophantine equations over the nonnegative integers \cite[Lem.~III.2.1]{davidson}, AF algebra theory is intimately tied to classical number-theoretic problems dating back to the work of Frobenius and Sylvester (we are not the first to combine number theory with the study of AF algebras; see \cite{EffrosShen, Boca, Mundici, Eckhardt, Spielberg}. Also, see \cite{Jacelon2025}, where statistical information is also analyzed for families of C*-algebras formed by combinatorial data in the form of certain graph C*-algebras). 

Let $\M_n$ denote the algebra of $n \times n$ complex matrices.  The number of unital $*$-homomorphisms from $\M_{n_1}\oplus \cdots \oplus \M_{n_k}$ into $\M_{m_1}\oplus\cdots \oplus \M_{m_{\ell}}$, up to unitary equivalence, depends upon the number of ways each $m_i$ can be factored 
\begin{equation*}
m_i = x_1 n_1 + x_2 n_2 + \cdots + x_k n_k
\end{equation*}
as an element of the additive subsemigroup of $\Z_{\geq 0}=\{0,1,2,\ldots\}$ generated by $n_1,n_2,\ldots,n_k$;
here each $x_i \in \Z_{\geq 0}$.  These dimension-consistency requirements connect AF-algebras to simultaneous factorization problems in numerical semigroups.  

A \emph{numerical semigroup} is an additive subsemigroup of $\Z_{\geq 0}$ with finite complement \cite{Assi, Rosales}. 
Each such semigroup is generated by certain positive integers $n_1,n_2,\ldots,n_k$ with $\gcd(n_1,n_2,\ldots,n_k) = 1$.  
The \emph{Frobenius number} of a numerical semigroup is the largest number in its complement.

For each $r \geq 1$, define
\begin{equation*}
\A_r = \bigoplus_{j=1}^k \M_{g_r n_j},    
\end{equation*}
in which $g_r$ is a suitable rapidly increasing sequence of positive integers such that $g_1=1$ and $g_r \mid g_{r+1}$; for example, $g_r = r!$ is admissible.  
Bratteli diagrams encode information about AF algebras, such as their ideals and isomorphism class \cite{Bratteli72, Bratteli74, Bratteli78}. 
Let $\vec{n} = (n_1,n_2,\ldots,n_k)$ and $\vec{g}= (g_1,g_2,\ldots)$. With these data, we construct an ensemble $\eng$ of Bratelli diagrams as follows.  The elements of $\eng$ are Bratteli diagrams corresponding to a sequence of ``standard'' unital $*$-homomorphisms $\phi_r:\A_r\to\A_{r+1}$. The resulting AF algebras $(\bigcup_{r=1}^{\infty} \A_r)^-$ are  \emph{rapidly growing}.

If there is only one matrix algebra at each level (that is, if $k=1$), then one obtains uniformly hyperfinite (UHF) algebras \cite[Ex.~III.5.1]{davidson}.  In this case, there is essentially only one choice of unital *-homomorphism $\phi_r:\A_r\to\A_{r+1}$ at each level.  Things become much more interesting if there are more direct summands at each level. 
Consequently, we insist that $k\geq 2$ throughout what follows.  Then there are uncountably many Bratteli diagrams in $\eng$ and we may begin to study these diagrams, and the associated AF algebras, from a probabilistic perspective.

There is a natural probability measure on each ensemble $\eng$.  The rapid growth of $g_r$ ensures that the asymptotic properties of factorizations in numerical semigroups come into play when we regard $\eng$ as a probability space.
Fortunately, almost all asymptotic questions about factorizations in numerical semigroups can be answered via spline theory \cite{factorization2,factorization5}, a cornerstone of computer-aided design \cite{curry-schoenberg}. Curry--Schoenberg B-splines are certain piecewise-polynomial functions with some desirable properties. They first arose 
as Peano's kernel in the theory of divided differences \cite{ReciprocalSchur, Davis, Phillips} and later in the study of P\'olya frequency functions (nonnegative functions that are log-concave on their support \cite{SchoenbergWhitney}).

The overarching theme of this paper is the use of spline theory to attack asymptotic combinatorial questions whose answers illuminate the class of rapidly growing AF algebras.  As such, this paper is organized as follows:

\begin{itemize}
    \item Section \ref{Section:Preliminaries} contains background information about Curry--Schoenberg B-splines, numerical semigroups, and AF algebras.
    
    \item We introduce rapidly growing AF algebras in Section \ref{Section:Rapidly}, where we begin our study of these algebras from a probabilistic viewpoint.  

    \item In Section \ref{Section:Injective}, we prove that with probability $1$ a random Bratteli diagram drawn from an ensemble $\eng$ has all but finitely many homomorphisms injective: this justifies our treatment of $\eng$ as a collection of AF algebras.  
    
    \item We consider simplicity in Section \ref{Section:Simplicity} and prove that under certain growth conditions on $\vec{g}$, a random algebra in $\eng$ is simple with probability $1$.
    
    \item Section \ref{Section:Ideals} provides a good reason to study ideals in rapidly growing AF algebras: they tell us that each $\eng$ contains several non-isomorphic UHF algebras.
    
    \item A more detailed study of the isomorphism problem is considered in Section \ref{Section:K-Theory}, where we use K-theory to prove that there are uncountably many distinct isomorphism classes of AF algebras present in certain ensembles $\eng$.

    \item We conclude in Section \ref{Section:Further} with several open problems that we hope will spur further research.  In particular, we believe that there is much more to be done in the study of rapidly growing AF algebras.
\end{itemize}
          
\noindent\textbf{Acknowledgments.} 
We thank   Gabe Udell for several helpful comments.
KA was supported by NSF Grant DMS-2316892.
SRG was supported by NSF Grant DMS-2452084.  SRG, CM, and EK were supported by NSF Grant DMS-2054002. 

\section{Preliminaries}\label{Section:Preliminaries}
This section contains preliminary material about B-splines (Subsection \ref{Subsection:BSpline}),
numerical semigroups (Subsection \ref{Subsection:Numerical}), and AF algebras (Subsection \ref{Subsection:AF}).
Our main goal is to set up the construction and study of rapidly growing AF algebras in Section \ref{Section:Rapidly}.
We go through this background material in some detail since these three threads seem not to have intersected in the past.  Indeed, B-splines are fundamental tools in computer-aided design, numerical semigroups are staple objects in combinatorics, and AF algebras are fundamental examples of operator algebras.

\subsection{B-splines}\label{Subsection:BSpline}
For a sequence of $k$ real \emph{knots} $a_1\leq a_2 \leq \cdots \leq a_k$, not all equal, the Curry--Schoenberg \emph{B-spline} is the piecewise-polynomial function 
\begin{equation*}
M(x;a_1,a_2,\ldots,a_k)=(k-1) M_{1,k-1}(x),
\end{equation*}
in which $M_{i,\ell}(x)$ is defined recursively by
\begin{equation}\label{BsplineBaseCase}
M_{i,1}(x)=
\begin{cases} 
(a_{i+1}-a_i)^{-1} & \text{if $a_i\leq x <  a_{i+1}$},\\ 
0 & \text{otherwise},
\end{cases}
\end{equation}
with
\begin{equation}\label{BsplineRecursion}
M_{i,\ell}(x)=\frac{x-a_i}{a_{i+\ell}-a_i}M_{i,\ell-1}(x)+\frac{a_{i+\ell}-x}{a_{i+\ell}-a_i}M_{i+1,\ell-1}(x) 
\end{equation}
and
\begin{equation}\label{M to M}
    M_{i,\ell}(x)=\frac{1}{\ell}M(x;a_i,a_{i+1},\dots, a_{i+\ell}),
\end{equation} 
as shown in \cite[eq.~(8)]{deBoorRecursion}.
If $a_{i+1}=a_i$, then $M_{i,1}(x)=0$.
The Curry--Schoenberg B-spline $M(x;a_1,a_2,\ldots,a_k)$ is a probability density function supported on $[a_1,a_k]$ that is positive on $(a_1,a_k)$ \cite[Thm.~1, eq.~(1.6)]{curry-schoenberg}.
Although tedious to compute manually, B-splines are implemented in most computer algebra systems; see Figure \ref{Figure:SplineDefinition}.

\begin{figure}
\centering
\begin{subfigure}{0.475\textwidth}
  \centering
  \includegraphics[width=\textwidth]{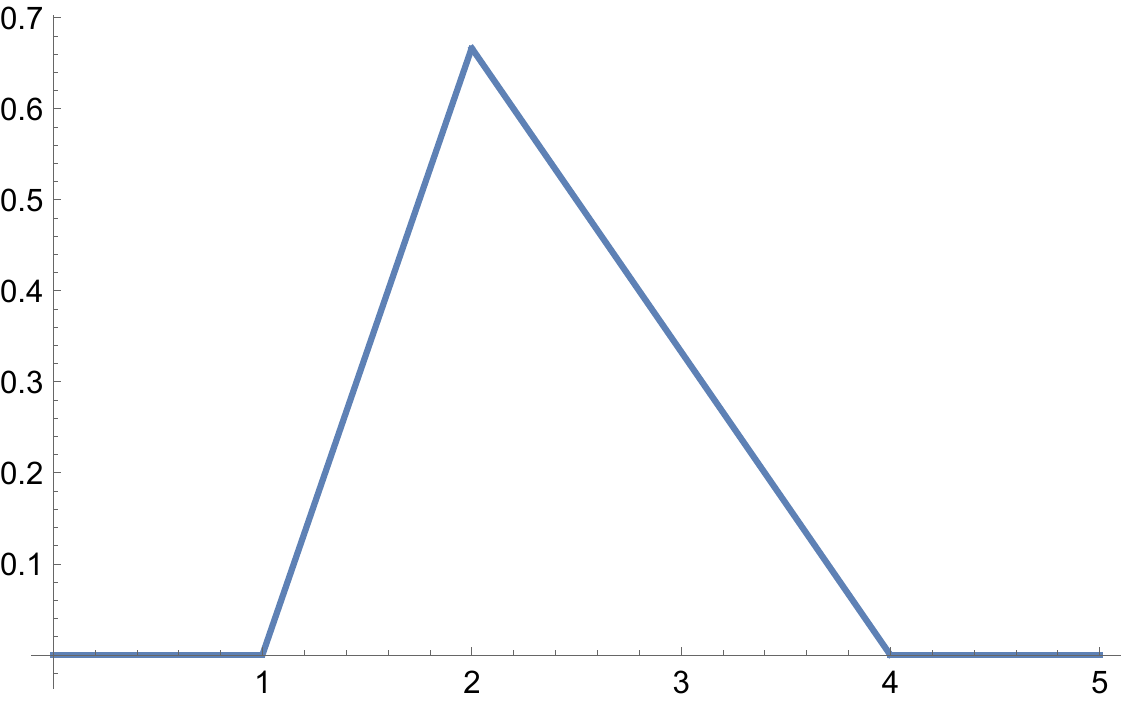}
  \caption{$M(x;1,2,4)$}
\end{subfigure}
\hfill
\begin{subfigure}{0.475\textwidth}
  \centering
  \includegraphics[width=\textwidth]{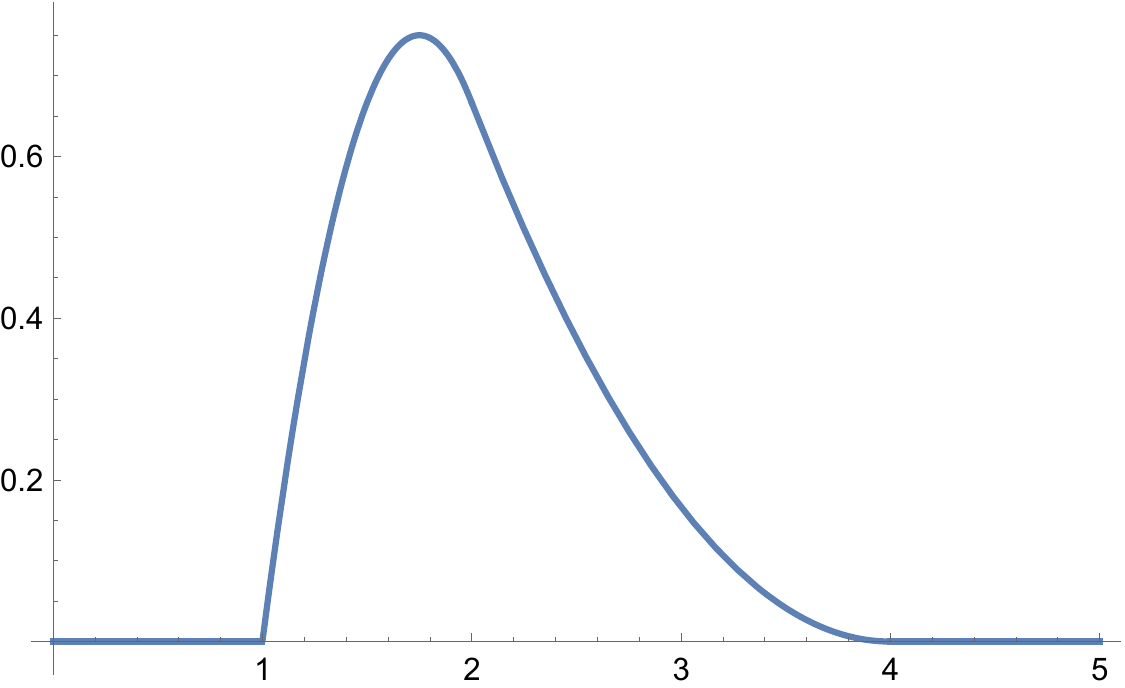}
  \caption{$M(x;1,1,2,4)$}
\end{subfigure}
\\
\begin{subfigure}{0.475\textwidth}
  \centering
  \includegraphics[width=\textwidth]{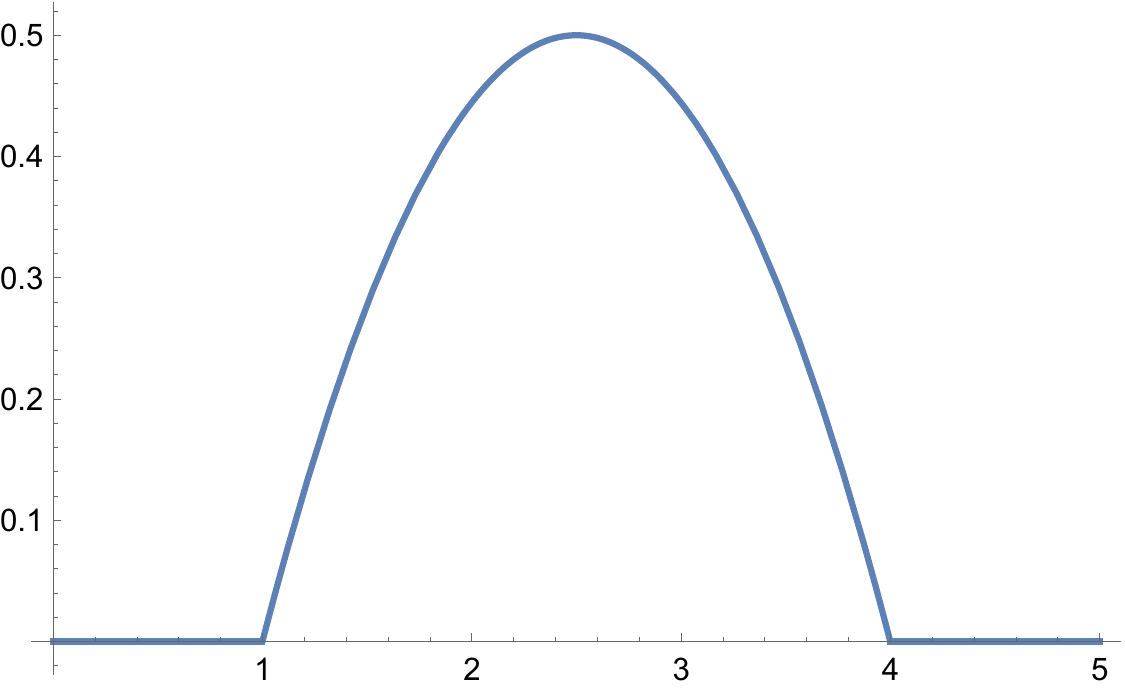}
  \caption{$M(x;1,1,4,4)$}
\end{subfigure}
\hfill
\begin{subfigure}{0.475\textwidth}
  \centering
  \includegraphics[width=\textwidth]{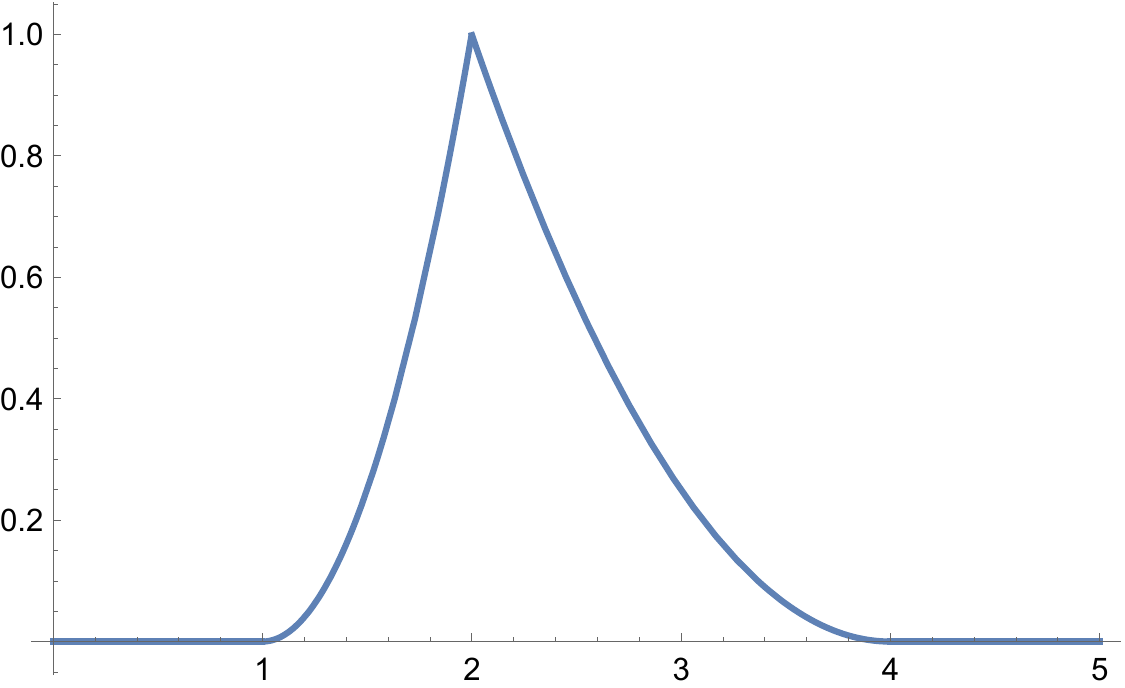}
  \caption{$M(x;1,2,2,4)$}
\end{subfigure}
  \\
\begin{subfigure}{0.475\textwidth}
  \centering
  \includegraphics[width=\textwidth]{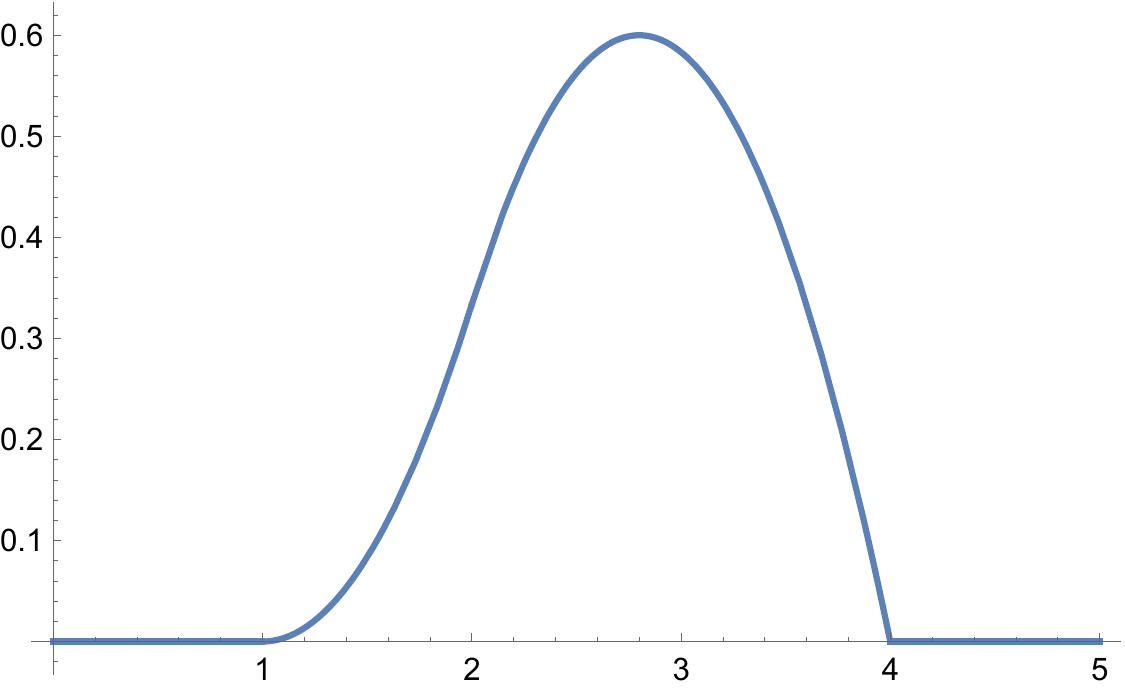}
  \caption{$M(x;1,2,4,4)$}
\end{subfigure}
\hfill
\begin{subfigure}{0.475\textwidth}
  \centering
  \includegraphics[width=\textwidth]{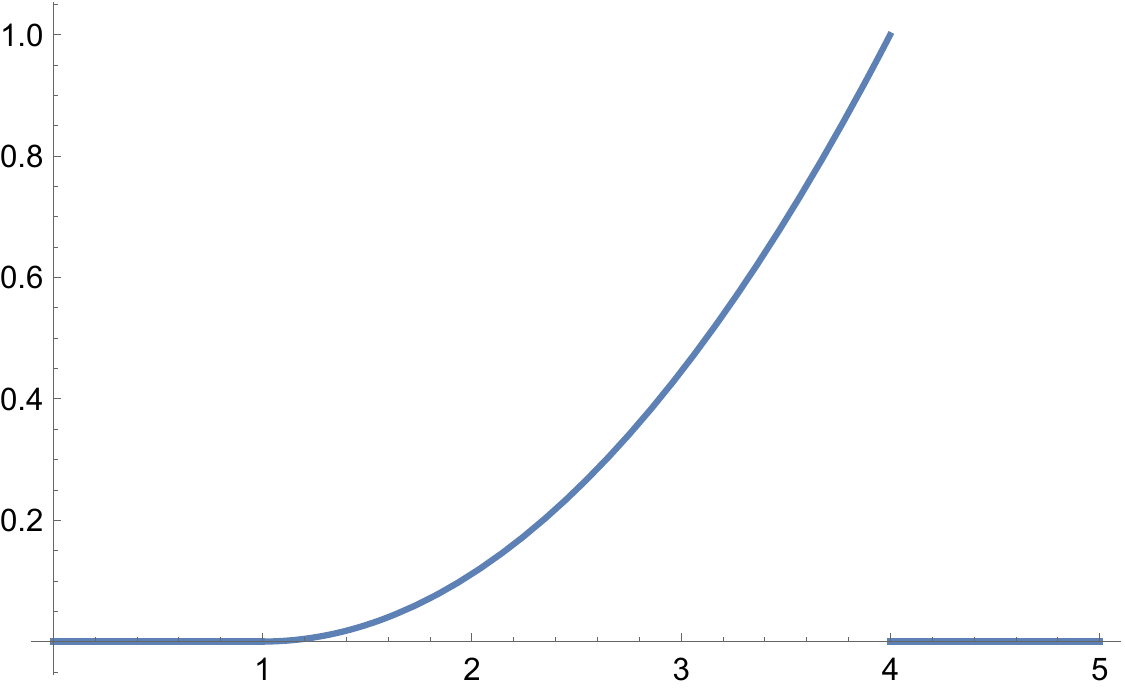}
  \caption{$M(x;1,4,4,4)$}
  \end{subfigure}
\caption{Plots of B-splines for various parameters.  Smoothness across knots is determined by the number of knots and their multiplicities.}
\label{Figure:SplineDefinition}
\end{figure}

For distinct knots $a_1<a_2<\cdots< a_k$, the Curry--Schoenberg B-spline is
\begin{equation}\label{eq:ExplicitSpline}
    M(x;a_1,a_2,\dots, a_k)=(k-1)\sum_{i=1}^k \frac{(a_i-x)_+^{k-2}}{\prod_{j\neq i}(a_i-a_j)},
\end{equation}
in which $x_+=\max\{x,0\}$ \cite[eq.~(1.4)]{curry-schoenberg}. In this case, the B-spline is polynomial of degree $k-2$ for $x\in (a_i,a_{i+1})$ and 
is $k-3$ times differentiable at $x=a_i$. 

If the knots are not distinct, then the corresponding B-spline is still polynomial of degree $k-2$ between knots but is of continuity class $C^m$ at $x=a_i$, where $m={k-2- | \{j:a_i=a_j\} |}$  \cite[Lem.~1]{curry-schoenberg}.  
There are different conventions in different sources, so it is important to remember that we consider splines of degree $k-2$ and that our knot indices start at $1$. For example, in \cite{curry-schoenberg}, the knot indices start at $0$.

If $a_1=a_2= \cdots = a_{k-1} = a$ and $a_k = b > a$, then \cite[(3.5)]{curry-schoenberg} yields
\begin{equation}\label{eq:SplineSame}
    M(x;\underbrace{a,a,\ldots,a}_{\text{$k-1$ times}},b) = 
    \begin{cases}
        \dfrac{(k-1)(b-x)^{k-2}}{(b-a)^{k-1}} & \text{if $a \leq x \leq b$},\\
        0 & \text{otherwise}.
    \end{cases}
\end{equation}

The ``B'' in ``B-spline'' stands for `basis' because the vector space of piecewise-polynomial functions on $\R$ with specified degree, breakpoints, and smoothness at the breakpoints has a basis comprised of B-splines \cite[Thm.~44]{practicalGuide}. 
Given knots $a_1\leq a_2\leq \cdots \leq a_k$, not all equal, the corresponding B-spline is the unique probability density function that is polynomial of degree $k-2$ between the knots, is 
\begin{equation*}
k-2-|\{j:a_i=a_j\}|    
\end{equation*}
times differentiable at each $x=a_i$, and is supported on $[a_1,a_k]$ \cite[p.~9]{Handbook}.

\begin{remark}\label{Remark:PointMass}
    If $a_1 = a_2 = \cdots = a_k = a$, then $M(x;a,a,\ldots,a)$, in which there are $k$ copies of $a$ listed, is undefined.
    However, \eqref{eq:SplineSame} ensures that $M(x;a,a,a,\ldots,a+\epsilon)$, in which there are $k-1$ copies of $a$ listed before the $a+\epsilon$, is a probability density function 
    supported on $[a,a+\epsilon]$ for each $\epsilon > 0$.  As $\epsilon\to0^+$, it converges
    weakly to the unit point measure $\delta_a$ at $a$.  Consequently, we adopt the convention
    \begin{equation*}
        M(x;\underbrace{a,a,\ldots,a}_{\text{$k$ times}}) = \delta_a.
    \end{equation*}
\end{remark}

\begin{remark}
Although the assumption $a_1\leq a_2 \leq \cdots \leq a_k$ is standard in the study of B-splines,
for convenience we permit knots to be unordered.  That is, if $a_1,a_2,\ldots,a_k \in \R$ are not all identical, we define
\begin{equation}\label{eq:SplineUnordered}
M(x;a_1,a_2,\ldots,a_k) = M(x;a_{\pi(1)},a_{\pi(2)},\ldots,a_{\pi(k)}),
\end{equation}
in which $\pi$ is a permutation of $\{1,2,\ldots,k\}$ such that $a_{\pi(1)} \leq a_{\pi(2)} \leq \cdots \leq a_{\pi(k)}$.  
\end{remark}

B-splines are built into many computer software packages, sometimes with different normalizations.
The images in Figure \ref{Figure:SplineDefinition} were produced with \texttt{Mathematica} via 
$\mathtt{BSplineBasis}[\{k-2, \{ a_1,a_2,\ldots,a_k\}\},0,x]$, which is then normalized to yield a probability density function. To obtain an explicit piecewise-polynomial representation of a B-spline, use the \texttt{PiecewiseExpand} command. For example,
\begin{equation*}
    M(x;1,2,2,4,4,4) =
    \begin{cases}
 \frac{5}{27} (x-1)^4 & \text{if $1\leq x<2$}, \\[2pt]
 \frac{5}{432} (x-4)^2 (43 x^2-152 x+136) & \text{if $2\leq x\leq 4$}, \\[2pt]
 0 & \text{otherwise}.
    \end{cases}
\end{equation*}
Consequently, B-splines should be regarded as concrete functions that can easily be evaluated when the need arises.

\subsection{Numerical semigroups}\label{Subsection:Numerical}
By \emph{semigroup}, we mean an additive subsemigroup of $\Z_{\geq 0} = \{0,1,2,\ldots\}$.
We write $S = \semigroup{\cdot}$ to mean that the indicated list of integers generates the semigroup $S$.
A \emph{numerical semigroup} is a semigroup with finite complement in $\Z_{\geq 0}$ \cite{Assi, Rosales}.  
A semigroup $S$ is numerical if and only if $\gcd S = 1$ \cite[Lem.~2.1]{Rosales}. 
Every semigroup in $\Z_{\geq 0}$ is isomorphic to a numerical semigroup.  Indeed, if $S$ is a semigroup with $d = \gcd S \geq 2$, then 
\begin{equation*}
S/d = \{ s/d : s\in S\}    
\end{equation*}
is a numerical semigroup \cite[Prop.~2.2]{Rosales}.

Each numerical semigroup $S$ has a unique minimal system of \emph{generators} 
$1 \leq n_1 < n_2 < \cdots < n_k$ 
such that $\gcd(n_1,n_2,\ldots,n_k) = 1$ and
$S = \langle n_1,n_2,\ldots,n_k\rangle$ \cite[Thm.~2.7]{Rosales}.
The largest natural number not in $S$ is the \emph{Frobenius number} $F(S)$ of $S$ \cite{RamirezDiop}. 
For example, the Frobenius number of $\semigroup{4,5} = \{0,4,5,8,9,10,12,13,14,15,\ldots\}$ is $11$. 
The asymptotic behavior of the Frobenius number has been studied, for example, by V.~I.~Arnold~\cite{arnold}, J.~Bourgain and Ya.~G.~Sinai~\cite{bourgainsinai}, and many others~\cite{alievhenkaicke,RNS}.

A \emph{factorization} of $n \in S = \semigroup{n_1,n_2,\ldots,n_k}$ is an $\vec{x} = [x_i] \in \Z_{\geq 0}^k$ such that 
\begin{equation*}
n = x_1n_1 + x_2 n_2 +\cdots + x_k n_k,
\end{equation*}
that is, $\vec{n}\cdot  \vec{x} = n$, in which $\vec{n} = [n_i] \in \Z_{\geq 0}^k$; our vectors, denoted in boldface, are column vectors.
The set of all factorizations of $n$ is  
\begin{equation}\label{eq:ZZZSn}
\ZZ_S(n) = \{ \vec{x} \in \Z_{\geq 0}^k :  \vec{n}\cdot  \vec{x} = n \}.
\end{equation}
Factorizations in numerical semigroups arise in discrete optimization via knapsack problems~\cite{pisinger1998knapsack,de2013algebraic} and in algebraic geometry and commutative algebra~\cite{abhyankar1967local,barucci1997maximality}.

\begin{remark}
We permit repeated generators and non-minimal generating sets.
If $n_i = n_j$ and $i \neq j$, then we regard $n_i$ and $n_j$ as being different colors.
For example, we regard $\semigroup{2,2,3}$ as being generated by a ``red'' $2$, a ``blue'' $2$, and an uncolored $3$.
Thus, we regard the three factorizations (best viewed in color)
$4 = {\color{red}2} + {\color{red} 2} = {\color{red}2} + {\color{blue} 2} = {\color{blue}2} + {\color{blue} 2}$ of the number $4$ as distinct. 
\end{remark}

We let $|\cdot|$ denote the cardinality of the indicated set (whether we mean absolute value or cardinality will be clear from context).  
If $S$ is a numerical semigroup, then $\ZZ_S(n) \neq \varnothing$ for all $n > F(S)$; 
the quantity $|\ZZ_S(n)|$ is sometimes called the \emph{Sylvester denumerant} \cite{RamirezDiop}.
If $d = \gcd S\geq 2$, then $\ZZ_S(n) = \varnothing$ whenever $d \nmid n$. If $d \mid n$, then $\ZZ_S(n)$ is nonempty whenever $n/d > F(S/d)$.

A \emph{weighted factorization length} is an expression of the form 
$\vec{m}\cdot  \vec{x}$,
in which $\vec{m} \in \R^k$.  
If $\vec{m} = \vec{1}$ is the all-ones vector, we obtain the \emph{length} 
$\ell(\vec{x})=\vec{1}\cdot \vec{x} = x_1+x_2+\cdots+x_k$ of the factorization $\vec{x}$.

\begin{example}
The \emph{McNugget semigroup} is $S = \semigroup{6, 9, 20}$, so named because McDonalds originally served
chicken nuggets in boxes of size $6$, $9$, and $20$.  We have
\begin{align*}
\ZZ_S(132) 
& = \{(2, 0, 6), (0, 8, 3), (3, 6, 3), (6, 4, 3), (9, 2, 3), (12, 0, 3), (1, 14, 0),\\
& \phantom{{}= \{} (4, 12, 0), (7, 10, 0), (10, 8, 0), (13, 6, 0), (16, 4, 0), (19, 2, 0), (22, 0, 0)\}
\end{align*}
These factorizations have lengths $8, 11, 12, 13, 14, 15, 15, 16, 17, 18, 19, 20, 21, 22$, respectively.
Note that $15$ occurs twice since $132$ has two factorizations of length $15$.
\end{example}

In what follows, $\sim$ denotes asymptotic equivalence; that is,
$f(n)\sim g(n)$ means that $\lim_{n\to\infty} f(n) / g(n) = 1$
(an implicit assumption is that the denominator does not vanish for sufficiently large $n$).
In a similar manner, the notation $f(n) \lesssim g(n)$ means that $\limsup_{n\to\infty} f(n)/g(n) \leq 1$. 
We make note of the well-known asymptotic formula for $|\ZZ_S(n)|$ when $S$ is a numerical semigroup \cite{schur1926additiven,comtet,factorization2,NathansonParts}:
\begin{equation}\label{eq:Cardinal}
|\ZZ_S(n)|  \sim \frac{n^{k-1}}{(k-1)! (n_1 n_2 \cdots n_k)} .
\end{equation}
See Remark \ref{Remark:GCD} for the general case.

The next result is from \cite{factorization5}; it is a generalization of the main theorem of \cite{factorization2}, which required distinct knots and used \eqref{eq:ExplicitSpline} instead of splines.
It was during the writing of \cite{WeightedMeans} when G.~Olshanski pointed out the connection between \eqref{eq:ExplicitSpline} and B-splines. Similar formulas go back to Peano's study of divided differences \cite[Chap.~III, Sec. 3.7]{DavisInterpolation}.  The statement of the following theorem is much more precise in \cite{factorization5} than we state here (one can give explicit error bounds; that is, with no implied constants left unspecified).  However, we suppress the superfluous details since for our purposes we have no need for explicit error bounds.

\begin{theorem}\label{Theorem:SemigroupWeak}
Let $S = \langle n_1,n_2,\ldots,n_k \rangle \subseteq \Z_{\geq 0}$ with 
\begin{equation*}
1\leq n_1 \leq n_2\leq \cdots \leq n_k
\quad \text{and} \quad \gcd(n_1, n_2, \dots,n_k)=1. 
\end{equation*}
If $\vec{m}= [m_i]  \in \R^k$ and $\n = [n_i]$ are linearly independent, and $f:\R\to\R$ is bounded and continuous, then
\begin{equation*}
 \frac{1}{|\ZZ_S(n)|} \sum_{ \vec{x} \in \ZZ_S(n)} f \Big( \frac{\vec{m}\cdot  \vec{x} }{n} \Big) \sim
 \int_{-\infty}^{\infty} \, f(t) M\Big(t; \frac{m_1}{n_1}, \frac{m_2}{n_2}, \ldots, \frac{m_k}{n_k} \Big)\,dt
\end{equation*}
as $n \to \infty$, in which $M\big(t; \frac{m_1}{n_1}, \frac{m_2}{n_2}, \ldots, \frac{m_k}{n_k} \big)$ is a Curry--Schoenberg B-spline.
\end{theorem}

If $\vec{m} = \vec{1}$ and $f(x) = x^d$ for some $d \in \Z_{\geq 0}$, then we obtain the asymptotic $d$th moment of the factorization lengths in a numerical semigroup \cite[eq.~(5)]{factorization2}:
\begin{equation*}\label{eq:Moments}
\frac{1}{|\ZZ_S(n)|}
\sum_{\vec{x} \in \ZZ_S(n)} \ell(\vec{x})^d
\sim \binom{p+k-1}{p}^{\!-1}\!\! h_p\bigg( \frac{1}{n_1}, \frac{1}{n_2} ,\ldots, \frac{1}{n_k} \bigg) n^p,
\end{equation*}
in which 
\begin{equation}\label{eq:CHS}
h_d(x_1,x_2,\ldots,x_n) = \sum_{1 \leq i_1 \leq \cdots \leq i_{d} \leq n} x_{i_1} x_{i_2}\cdots x_{i_d}
\end{equation}
is the \emph{complete homogeneous symmetric polynomial} (CHS polynomial) of degree $d$ in the $k$ variables $x_1,x_2,\ldots,x_k$ \cite[Sec.~7.5]{StanleyBook2}.  It is the sum of all monomials of degree $d$ in $x_1,x_2,\ldots,x_n$.  For example,
the first several CHS polynomials in $x_1$ and $x_2$ are 
$h_0(x_1,x_2)= 1$,
$h_1(x_1,x_2)= x_1+x_2$, and
$h_2(x_1,x_2)= x_1^2+x_1 x_2+x_2^2$.
An important special case of \eqref{eq:Moments} gives the asymptotic mean factorization length:
\begin{equation}\label{eq:Mean}
\frac{1}{|\ZZ_S(n)|} \sum_{\vec{x} \in \ZZ_S(n)} \ell(\vec{x})
\sim \frac{n}{k}\bigg(\frac{1}{n_1} + \frac{1}{n_2} + \cdots + \frac{1}{n_k} \bigg).
\end{equation}

\begin{remark}
In 1977, D.B.~Hunter proved that the CHS polynomials of even degree are positive-definite functions: they vanish only at the origin and are otherwise positive \cite{Hunter}.  This has been reproved many times, for example in \cite[Thm.~1]{CHS-Norms}, \cite[Thm.~2]{WeightedMeans}, \cite[Cor.~17]{factorization2}, and Tao's blog \cite[Thm.~1]{Tao}; see \cite{bouthat2024hunterspositivitytheoremrandom} for more proofs and references.  Hunter's theorem is also the starting point for the developing theory of random vector norms \cite{CHS-Norms,RVN1,RVN2}.  In another direction,
the second named author and J.~Vol\v{c}i\v{c} recently proved a noncommutative generalization of Hunter's result, in which the scalar variables are replaced with hermitian operators \cite{garcia2025noncommutativegeneralizationhunterspositivity}.
\end{remark}

\begin{remark}\label{Remark:LocalLimit}
Theorem \ref{Theorem:SemigroupWeak} says that the probability measures
\begin{equation*}
    \nu_n = \frac{1}{| \ZZ_S(n) |} \sum_{\vec{x} \in \ZZ_S(n)} \delta_{\frac{ \vec{m}\cdot \vec{x}}{n}}
\end{equation*}
converge weakly to the absolutely continuous measure $\nu$ with density function
$M(x; \frac{m_1}{n_1}, \frac{m_2}{n_2}, \ldots, \frac{m_k}{n_k})$.  All of the measures involved are supported in the same compact interval $[a,b]$, in which $a = \min_{1\leq i \leq k} m_i/n_i$ and $b = \max_{1\leq i \leq k} m_i / n_i$.  Since $\nu$ does not assign positive measure to any point, $\nu_n(I)\to \nu(I)$ for any interval $I \subseteq \R$ \cite[Thm.~25.8]{Billingsley}, we obtain the asymptotic version of \cite[Thm.~C]{factorization5}:
\begin{equation}\label{eq:LocalLimit}
     \frac{1}{|\ZZ_S(n)|} \sum_{ \substack{ \vec{x} \in \ZZ_S(n) \\ (\vec{m}\cdot \vec{x})/n \in I }} f \Big( \frac{\vec{m}\cdot \vec{x} }{n} \Big) \sim
 \int_I \, f(t) M\Big(t; \frac{m_1}{n_1}, \frac{m_2}{n_2}, \ldots, \frac{m_k}{n_k} \Big)\,dt.
\end{equation}
\end{remark}

\begin{remark}\label{Remark:AllSame}
    If $\vec{m} = \vec{n} = \vec{1}$, then Theorem \ref{Theorem:SemigroupWeak} does not immediately apply.  However, \eqref{eq:Moments} and \eqref{eq:Mean} hold nevertheless.  Indeed, one can
    perturb $\vec{m}$ to satisfy the hypotheses of Theorem \ref{Theorem:SemigroupWeak}, then appeal to Remark \ref{Remark:PointMass} and weak convergence.  The same occurs in \eqref{eq:Cardinal}, and this can be worked out directly.  If $\vec{n} = \vec{1}$, then the factorizations of $n$ are the weak compositions of $n$ into $k$ parts.  A routine stars-and-bars combinatorial argument tells us that there are 
    \begin{equation*}
        \binom{n+k-1}{k-1} 
        = \frac{(n+k-1)(n+k)\cdots(n+1)}{(k-1)!}
        \sim \frac{n^{k-1}}{(k-1)! (1\cdot1 \cdots 1)}
    \end{equation*}
    many of them as $n \to \infty$, in accordance with \eqref{eq:Cardinal}.  
\end{remark}

\begin{remark}\label{Remark:GCD}
    If $d= \gcd(n_1,n_2,\ldots,n_k) \neq 1$, then Theorem \ref{Theorem:SemigroupWeak} holds in a modified form.
    First observe that $\ZZ_S(n) = 0$ whenever $d \nmid n$.  Then apply Theorem \ref{Theorem:SemigroupWeak} to 
    $S/d = \semigroup{ \frac{n_1}{d}, \frac{n_2}{d},\ldots, \frac{n_k}{d}}$ in which $d\mid n$ and each $n_i$ is replaced with $n_i/d$. 
    Then
    \begin{equation}
    \label{eq:Z-gcd-change}
        \ZZ_S(n)=\ZZ_{S/d}(n/d),
    \end{equation}
    so as $n\to\infty$ with $d\mid n$, \eqref{eq:Cardinal} ensures that
    \begin{equation}\label{eq:Cardinal-with-d}
    |\ZZ_S(n)|=|\ZZ_{S/d}(n/d)| 
\sim \frac{ (\frac{n}{d})^{k-1}}{(k-1)!\frac{n_1}{d} \frac{n_2}{d} \cdots \frac{n_k}{d}}  
=\frac{ dn^{k-1}}{(k-1)!(n_1n_2\cdots n_k)}  .
    \end{equation}
\end{remark}

The following corollary generalizes Theorem \ref{Theorem:SemigroupWeak} and \eqref{eq:LocalLimit}
in a manner suitable for our needs:
we need not assume that the generators are relatively prime.

\begin{corollary}\label{Corollary:SemigroupWeak}
Let $S = \langle n_1,n_2,\ldots,n_k \rangle \subseteq \Z_{\geq 0}$ with $1\leq n_1 \leq n_2\leq \cdots \leq n_k$ and $d=\gcd(n_1,n_2,\ldots,n_k)$. Let $S/d = \semigroup{ \frac{n_1}{d}, \frac{n_2}{d},\ldots, \frac{n_k}{d}}$. If $\vec{m}= [m_i]  \in \R^k$ and $\n = [n_i]$ are linearly independent, $I \subseteq \R$ is an interval, $dI$ is its $d$-dilation, and $f:\R\to\R$ is bounded and continuous, then
\begin{equation*}\label{eq:LocalLimit-gcd-d}
 \frac{1}{|\ZZ_S(n)|} \sum_{ \substack{ \vec{x} \in \ZZ_S(n) \\ (\vec{m}\cdot \vec{x})/n \in I }} f \Big( \frac{\vec{m}\cdot  \vec{x} }{n} \Big) \sim
 \int_{dI} \, f(t/d) M\Big(t; \frac{m_1d}{n_1}, \frac{m_2d}{n_2}, \ldots, \frac{m_kd}{n_k} \Big)\,dt
\end{equation*}
as $n\to\infty$ with $d \mid n$,
in which $M\big(t; \frac{m_1d}{n_1}, \frac{m_2d}{n_2}, \ldots, \frac{m_kd}{n_k} \big)$ is a Curry--Schoenberg B-spline.
\end{corollary}

\begin{proof}
    Observe that \eqref{eq:Z-gcd-change} implies that
    \begin{equation*}
         \frac{1}{|\ZZ_S(n)|} \sum_{ \substack{ \vec{x} \in \ZZ_S(n) \\ (\vec{m} \cdot \vec{x})/n \in I}}  f \Big( \frac{\vec{m}\cdot  \vec{x} }{n} \Big)
         =
         \frac{1}{|\ZZ_{S/d}(n/d)|} \sum_{ \substack{\vec{x} \in \ZZ_{S/d}(n/d) \\ \frac{ \vec{m} \cdot \vec{x}}{n/d} \in dI} } f \Big( \frac{\vec{m}\cdot \vec{x} }{d(n/d)} \Big) .
    \end{equation*}
    Now apply Theorem \ref{Theorem:SemigroupWeak} and appeal to Remark \ref{Remark:LocalLimit}.
\end{proof}

\subsection{AF algebras}\label{Subsection:AF}
We denote by $\M_{m\times n}(\cdot)$ the set of $m \times n$ matrices with entries in the indicated set.  Let $\M_n$ denote the C*-algebra of $n\times n$ complex matrices. Every finite-dimensional C*-algebra is *-isomorphic to a direct sum of matrix algebras \cite[Thm.~III.1.1]{davidson}; in particular, each such algebra is unital.

Given unital C*-algebras $\A$ and $\B$, we denote the set of unital *-homomorphisms from $\A$ to $\B$ by $\hom(\A, \B)$.
As an abuse of language, we sometimes drop the asterisk in *-isomorphic and simply call two C*-algebras \emph{isomorphic}.  Moreover, we often use the word ``algebra'' as an abbreviation of ``C*-algebra.''
If 
\begin{equation*}
\A=\bigoplus_{i=1}^k \M_{n_i} \qquad\text{and}\qquad \B=\bigoplus_{j=1}^{\ell} \M_{m_j},     
\end{equation*}
then $\phi\in \hom(\A, \B)$ is \emph{standard} if for each $(a_1,a_2, \ldots,a_k) \in \A$, we have
\begin{equation*}
\phi((a_1,a_2, \ldots,a_k))=(b_1, b_2, \ldots, b_{\ell}),
\end{equation*}
in which each $b_j \in \M_{m_j}$ is block diagonal with diagonal blocks among $a_1,a_2,\ldots,a_k$ listed in increasing order of subscript. 
For example, the unital *-homomorphism $\phi: \M_2\oplus \M_4 \to \M_6\oplus \M_4 $ defined by
\begin{equation}\label{eq:standard}
\phi(a_1, a_2) = \left( \begin{bmatrix}
    a_1 & \\
    & a_2
\end{bmatrix}, 
\begin{bmatrix}
    a_1 & \\
    & a_1
\end{bmatrix}
\right)
\end{equation}
is standard.  The set $\shom(\A,\B)$ of standard, unital *-homomorphisms is finite.

Every $\psi\in \hom(\A, \B)$ is unitarily equivalent to some $\phi\in\shom(\A, \B)$ \cite[Lemma III.2.1]{davidson}. 
Each $\phi\in \shom(\A, \B)$ is uniquely determined by its \emph{matrix of partial multiplicities} 
$X=[x_{ji}] \in \M_{\ell\times  k}(\Z_{\geq 0})$, which satisfies
\begin{equation*}
    \sum_{i=1}^k x_{ji}n_i=m_j
\end{equation*}
for $1 \leq j \leq \ell$ \cite[Lemma III.2.1]{davidson}.
For the $\phi$ in \eqref{eq:standard}, we have  
$X = \big[ \begin{smallmatrix}
    1 & 1\\ 2 & 0
\end{smallmatrix}\big]$.

The \emph{number of times $\phi$ maps $\M_{n_i}$ into $\M_{m_j}$} is the partial multiplicity $x_{ji}$ (this is well defined, even if some of the $n_i$ or $m_j$ are repeated). For example,
in \eqref{eq:standard}, $x_{21} = 2$  is the number of times $\phi$ maps $\M_2$ into $\M_4$. 

The \emph{number of times $\M_{n_i}$ appears in $\phi(\A)$} is 
\begin{equation*}
X \vec{e}_j = \sum_{i=1}^k x_{ij} ,
\end{equation*} 
in which $\vec{e}_j$ denotes the $j$th standard basis vector.  It is the total number
of times $\phi$ maps $\M_{n_i}$ into any of the $\M_{m_j}$.
For the $\phi$ in \eqref{eq:standard}, $\M_2$ appears $3=1+2$ times in $\phi(\A)$. 
Note that we are not counting algebras up to isomorphism, but rather the number of copies of a matrix algebra that appear.

The \emph{number of matrix algebras in $\phi(\A)$} is \begin{equation*}
 \vec{1}\cdot  X \vec{1} = \sum_{j=1}^{\ell} \sum_{i=1}^k x_{ji}.
\end{equation*} 
For the $\phi$ in \eqref{eq:standard}, we see that $\phi(\A)$ contains $4$ matrix algebras.  It is important to note that we count by multiplicity and we do not identify isomorphic algebras.

An \emph{AF algebra} is a C*-algebra $\A$ that is a direct limit of a direct sequence of finite-dimensional C*-algebras $(\A_r, \phi_r)_{r=1}^{\infty}$, where each $\phi_r \in \hom( \A_r,\A_{r+1})$ \cite[Ch.~6]{Murphy}.
We denote the direct sequence by the diagram
\begin{equation*}
\A_1 \overset{\phi_1}{\longrightarrow}
\A_2 \overset{\phi_2}{\longrightarrow}
\A_3 \overset{\phi_3}{\longrightarrow}\cdots
\end{equation*}
and note that
\begin{equation*}
    \A=\overline{\bigcup_{r=1}^{\infty} \phi^r(\A_r)},
\end{equation*}
in which $\phi^r :\A_r \to \A$ is the natural map and $\phi^{r+1}\circ \phi_r = \phi^r$ for each $r\in \N$ \cite[p.~176, Ch.~6]{Murphy}. We refer to $r$ or $\A_r$ as a \emph{level} of the Bratteli diagram or associated AF algebra. 
The associated \emph{Bratteli diagram} of a direct sequence is described in \cite{Bratteli72}.

The weight of the edge between $n_j$ on the $r$th level and $n_i$ on the $(r+1)$th level of the Bratteli diagram equals the $(i,j)$ entry of the partial-multiplicity matrix.

\begin{example}
    For $\A_1=\M_{12}\oplus \M_{42}$ and $\A_2=\M_{48}\oplus \M_{168}$, the matrices
    \begin{equation*}
    \begin{bmatrix}
    4&0\\
    0&4
    \end{bmatrix}, 
    \qquad \begin{bmatrix}
        4&0\\
        7&2
    \end{bmatrix},\qquad
        \begin{bmatrix}
        4&0\\
        14&0
    \end{bmatrix},
    \end{equation*}
    correspond to the three potential partial Bratteli diagrams
\[\begin{tikzcd}
	12 & 48 & 12 & 48 & 12 & 48 \\
	42 & 168 & 42 & 168 & 42 & 168
	\arrow["4", from=1-1, to=1-2]
	\arrow["4", from=1-3, to=1-4]
	\arrow["7"{pos=0.4}, from=1-3, to=2-4]
	\arrow["4", from=1-5, to=1-6]
	\arrow["14"', from=1-5, to=2-6]
	\arrow["4", from=2-1, to=2-2]
	\arrow["2", from=2-3, to=2-4]
\end{tikzcd}\]
where, as is customary, each vertex reflects the dimension of the corresponding matrix algebra and
edge weights describe the multiplicity of an embedding.
\end{example}

There are certain moves on Bratteli diagrams that do not change the isomorphism class of the associated AF algebra \cite[p.~179]{lectures-on-op-theory}. The phrase \emph{all but infinitely many} refers to a set with infinite complement in some larger set that is understood from context. Two Bratteli diagrams give rise to isomorphic AF algebras if and only if they are related by a sequence of moves of the following kinds.
    \begin{enumerate}[leftmargin=*]
        \item Remove the first level.
        \item Remove all but infinitely many (non-inital) levels and appropriately compose the edges.
    \end{enumerate}
    We refer to the second move as \emph{telescoping}.
Recall that each standard unital *-homomorphism $\phi_r:\A_r\to\A_{r+1}$ corresponds to a partial-multiplicity matrix $X_r\in \M_k(\Z_{\geq 0})$. With this identification, given a sequence
\begin{equation*}
        \cdots \longrightarrow\mathcal A_{r-1}\overset{X_{r}}{\longrightarrow}\mathcal A_r\overset{X_{r+1}}{\longrightarrow}\mathcal A_{r+1}\longrightarrow\cdots,
\end{equation*}
the second move yields
\begin{equation*}
        \cdots \longrightarrow\mathcal A_{r-1}\overset{X_{r+1}X_{r}}{\longrightarrow}\mathcal A_{r+1}\longrightarrow\cdots.
\end{equation*}

\begin{example}
We apply the second move to the second and third levels of the Bratteli diagram on the left.  This yields the Bratteli diagram on the right.
\begin{equation*}\begin{tikzcd}[cramped]
	{d_1} & {d_2} & {d_3} & {d_4} & \cdots & {d_1} &&& {d_4} & \cdots \\
	{e_1} & {e_1} & {e_3} & {e_4} & \cdots & {e_1} &&& {e_4} & \cdots
	\arrow["{{{x_{11}^{\scriptscriptstyle(1)}}}}", from=1-1, to=1-2]
	\arrow["{{{x_{11}^{(2)}}}}", from=1-2, to=1-3]
	\arrow["{{{x_{21}^{(2)}}}}"{description}, from=1-2, to=2-3]
	\arrow["{{{x_{11}^{(3)}}}}", from=1-3, to=1-4]
	\arrow[from=1-4, to=1-5]
	\arrow["{{{x_{11}^{(1)}x_{11}^{(2)}x_{11}^{(3)}}}}", from=1-6, to=1-9]
	\arrow["{{{x_{11}^{(1)}x_{21}^{(2)}x_{22}^{(3)}}}}"{description}, from=1-6, to=2-9]
	\arrow[from=1-9, to=1-10]
	\arrow["{{{x_{22}^{(1)}}}}"', from=2-1, to=2-2]
	\arrow["{{{x_{22}^{(2)}}}}"', from=2-2, to=2-3]
	\arrow["{{{x_{22}^{(3)}}}}"', from=2-3, to=2-4]
	\arrow[from=2-4, to=2-5]
	\arrow["{{{x_{22}^{(1)}x_{22}^{(2)}x_{22}^{(3)}}}}"', from=2-6, to=2-9]
	\arrow[from=2-9, to=2-10]
\end{tikzcd}\end{equation*}
The partial-multiplicity matrices on the left are
\begin{equation*}
X_1=\begin{bmatrix}
    x_{11}^{(1)}&0\\[2pt]
    0&x_{22}^{(1)}
\end{bmatrix},\quad
X_2=\begin{bmatrix}
    x_{11}^{(2)}&0\\[2pt]
    x_{21}^{(2)}&x_{22}^{(2)}
\end{bmatrix},\quad\text{and}\quad
X_3=\begin{bmatrix}
    x_{11}^{(3)}&0\\[2pt]
    0&x_{22}^{(3)}
\end{bmatrix}.
\end{equation*}
As expected, the new partial-multiplicity matrix on the right is
\begin{equation*}
\begin{bmatrix}
    x_{11}^{(1)}x_{11}^{(2)}x_{11}^{(3)}&0\\[3pt]
    x_{11}^{(1)}x_{21}^{(2)}x_{22}^{(3)}&x_{22}^{(1)}x_{22}^{(2)}x_{22}^{(3)}
\end{bmatrix}=X_3X_2X_1.
\end{equation*}
\end{example}

\section{Rapidly growing AF algebras}\label{Section:Rapidly}
In this section, we introduce novel families of AF algebras with appealing combinatorial and probabilistic properties.
We begin in Subsection \ref{Subsection:Setup} with the definition of these ``rapidly growing AF algebras.''
In Subsection \ref{Subsection:ProbabilityMeasure}, we introduce a probability measure on certain ensembles of rapidly growing AF algebras.  Then in Subsection \ref{Subsection:Stats} we examine the asymptotic statistical properties of these algebras. We demonstrate these ideas with several illustrative examples in Subsection \ref{Subsection:Examples}.

\subsection{Setup}\label{Subsection:Setup}
Let $k \geq 2$ and let $1 \leq n_1 \leq n_2 \leq \cdots \leq n_k$ be integers such that $\gcd(n_1,n_2,\ldots,n_k) = d$.
We refer to the (column) vector $\vec{n} = [n_i]\in \Z_{\geq 1}^k$ as a \emph{generating vector} and its entries as \emph{generators}.  Then $S = \semigroup{n_1,n_2,\ldots,n_k}$ is an additive semigroup in $\Z_{\geq 0}$. It is not necessarily numerical or minimally presented. 

Let $1=g_1 < g_2 < \cdots$ be an increasing sequence of integers such that each $g_r$ divides $g_{r+1}$ and such that the positive integer sequence $\gamma_r = g_{r+1} / g_r$ satisfies
\begin{equation}\label{eq:Rapid}
\sum_{r=1}^{\infty} \frac{1}{\gamma_r^k} < \infty.
\end{equation}
For example,  $g_r = r!$ and $g_r=2^{(r-1)^2}$ are admissible since
$\gamma_r = r+1$ and $\gamma_{r} = 2^{2r-1}$, respectively (recall that $k \geq 2$).
Since $\gamma_r$ is a sequence of positive integers, \eqref{eq:Rapid} holds whenever
$\gamma_r$ is strictly increasing.  For example, this occurs if $g_r$ is strictly logarithmically convex; that is, if $g_{r+1} g_{r-1} > g_r^2$.

Define 
\begin{equation*}
\A_r = \bigoplus_{i=1}^k \M_{g_r n_i}.
\end{equation*}
For each $r\geq 1$, the set $\hom(\A_r,\A_{r+1})$ of unital *-homomorphisms from $\A_r$ into $\A_{r+1}$
is nonempty since $g_r n_i \mid g_{r+1} n_i$ for each $1 \leq i \leq k$. 
Thus, the finite set $\shom(\A_r,\A_{r+1})$ of standard unital $*$-homomorphisms is nonempty.

Each $\phi_r \in  \shom(\A_r, \A_{r+1})$ corresponds to a partial-multiplicity matrix 
\begin{equation*}
X_r = \big[ x_{ij}^{(r)} \big]_{i,j=1}^k \in \M_k(\Z_{\geq 0})
\end{equation*}
such that $X_r (g_r \vec{n}) = g_{r+1} \vec{n}$; that is, 
\begin{equation}\label{eq:Xngn}
X_r \vec{n} = \gamma_r \vec{n}.
\end{equation}
This is equivalent to the $k \times k$ linear system 
\begin{equation}\label{eq:ADN}
\begin{matrix}
x^{(r)}_{11} n_1 &+& \cdots &+& x^{(r)}_{1k} n_k &=&  \gamma_r n_1 ,\\[3pt]
x^{(r)}_{21} n_1 &+& \cdots &+& x^{(r)}_{2k} n_k &=& \gamma_r n_2 ,\\[3pt]
\vdots & \vdots & \ddots & \vdots & \vdots & \vdots & \vdots\\
x^{(r)}_{k1} n_1 &+& \cdots &+& x^{(r)}_{kk} n_k &=&  \gamma_r n_k ,
\end{matrix}
\end{equation}
which describes certain factorizations of each $\gamma_r n_i$ in $S$.

When we speak of the \emph{$r$th level},
we refer to a standard unital *-homomorphism $\phi_r:\A_r\to\A_{r+1}$ and the associated partial-multiplicity matrix $X_r$. 
If $r$ is clear from context, we drop the subscript on $X_r$ and the superscripts on its entries.
We may partition $X$ according to its columns $\vec{c}_j \in \Z_{\geq 0}^k$ or its rows
$\vec{r}_i  \in \Z_{\geq 0}^k$.  That is,
\begin{equation}\label{eq:RowsColumns}
    X = [ \vec{c}_1~\vec{c}_2~\cdots~ \vec{c}_k] 
    = 
    \begin{bmatrix}
    \vec{r}_1^{\top}\\ \vdots \\ \vec{r}_k^{\top}    
    \end{bmatrix}
    \in \M_k(\Z_{\geq 0})    .
\end{equation}
The equations \eqref{eq:ADN} assume the form $\vec{r}_i\cdot \vec{n} = \gamma_r n_i$ for $1 \leq i \leq k$ and 
$x_{ij} = \vec{e}_i\cdot \vec{c}_j = \vec{r}_i\cdot  \vec{e}_j$, in which
$\vec{e}_1, \vec{e}_2,\ldots, \vec{e}_k$ are the standard basis vectors for $\R^k$.

We identify the elements $(\phi_1,\phi_2,\ldots)$ of the Cartesian product
\begin{equation}\label{eq:EngDef}
    \E(\vec{n}, \vec{g}) = \prod_{r=1}^{\infty} \shom(\A_r,\A_{r+1})
\end{equation}
with the corresponding direct sequence (called a \emph{chain} for short)
\begin{equation}\label{eq:DirectSequenceRapid}
    \A_1 \overset{\phi_1}{\longrightarrow}
\A_2 \overset{\phi_2}{\longrightarrow}
\A_3 \overset{\phi_3}{\longrightarrow}\cdots.
\end{equation}
As an abuse of language, we may identify the \emph{rapidly growing} AF algebra 
\begin{equation}\label{eq:RapidLimit}
\A = \overline{\bigcup_{n=1}^{\infty}\phi^r(\A_r)}
\end{equation}
with these data.  Note that two chains may be distinct even if the resulting AF algebras are isomorphic.
Since we plan to treat $\eng$ as a probability space (Subsection \ref{Subsection:ProbabilityMeasure}),
we refer to it as an \emph{ensemble}.  
We sometimes regard each such $\A$ as the direct sequence \eqref{eq:DirectSequenceRapid} or its associated Bratteli diagram, which we denote by $B(\A)$. We may also let $\A$ denote the associated direct limit \eqref{eq:RapidLimit}.

\begin{example}\label{Example:TwoGenerator1}
Fix $\vec{n} = \twovector{n_1}{n_2}$ and let $g_r = r!$.
Then $\A_r = \M_{ n_1 r!} \oplus \M_{n_2 r!}$.
As $r \to \infty$, there are many unital embeddings $\phi_r: \A_r \to \A_{r+1}$; see Figure \ref{Figure:TwoGenerator}.
In particular, $\eng$ contains infinitely many chains.  Do infinitely many non-isomorphic AF algebras arise in this manner?
What is the probability an AF algebra drawn from $\eng$ is simple?  What is the asymptotic expected number of matrix algebras in $\phi_r(\A_r)$?  These are the type of questions we hope to answer in what follows.
\end{example}

\begin{figure}
    \begin{subfigure}{1\textwidth}\label{subfigure:nonsimple}
    \centering
\[\begin{tikzcd}
	2 & 4 & 12 & 48 & 240 & 1440  \cdots\\
	3 & 6 & 18 & 72 & 360 & 2160  \cdots
	\arrow["2", from=1-1, to=1-2]
	\arrow["3", from=1-2, to=1-3]
	\arrow["4", from=1-3, to=1-4]
	\arrow["5", from=1-4, to=1-5]
	\arrow["6", from=1-5, to=1-6]
	\arrow["2", from=2-1, to=2-2]
	\arrow["3", from=2-2, to=2-3]
	\arrow["4", from=2-3, to=2-4]
	\arrow["5", from=2-4, to=2-5]
	\arrow["6", from=2-5, to=2-6]
\end{tikzcd}\]
\caption{Partial Bratteli diagram for one possible $\mathcal A\in \eng$.}
\end{subfigure}
\begin{subfigure}{1\textwidth}\label{subfigure:simple}
    \centering
\[\begin{tikzcd}
	2 & 4 & 12 & 48 & 240 & 1440 \cdots\\
	3 & 6 & 18 & 72 & 360 & 2160 \cdots
	\arrow["2", from=1-1, to=1-2]
	\arrow["3", from=1-2, to=1-3]
	\arrow["3"{description, pos=0.6}, curve={height=6pt}, from=1-2, to=2-3]
	\arrow["4", from=1-3, to=1-4]
	\arrow["3"{description, pos=0.6}, curve={height=6pt}, from=1-3, to=2-4]
	\arrow["2", from=1-4, to=1-5]
	\arrow["3"{description, pos=0.6}, curve={height=6pt}, from=1-4, to=2-5]
	\arrow["3", from=1-5, to=1-6]
	\arrow["3"{description, pos=0.6}, curve={height=6pt}, from=1-5, to=2-6]
	\arrow["2"', from=2-1, to=2-2]
	\arrow["1"', from=2-2, to=2-3]
	\arrow["2"', from=2-3, to=2-4]
	\arrow["2"{description, pos=0.6}, curve={height=-6pt}, from=2-4, to=1-5]
	\arrow["3"', from=2-4, to=2-5]
	\arrow["2"{description, pos=0.6}, curve={height=-6pt}, from=2-5, to=1-6]
	\arrow["4"', from=2-5, to=2-6]
\end{tikzcd}\]
\caption{Partial Bratteli diagram for another possible $\mathcal A'\in \eng$.}
\end{subfigure}
\caption{Partial Bratteli diagrams for two possible AF algebras coming from $\eng$, in which $\vec{n} = \twovector{2}{3}$ and $g_r=r!$. Arrow multiplicities are denoted by numbers above each arrow, rather than with multiple arrows. As $r$ increases, the number of possible unital embeddings $\phi_r:\A_{r}\to\A_{r+1}$ increases, hence so does the potential for different algebraic structure. Although $\A$ and $\A'$ belong to the same ensemble, their structure is quite different; $\A$ has two ideals, while $\A'$ is simple. }
\end{figure}\label{Figure:TwoGenerator}

\begin{remark}
    Many standard examples of AF algebras do not fit into the present framework.
    This is a positive: rapidly growing AF algebras provide new territory to explore.
    First note that the restriction $k\geq 2$ provides us with a certain amount of combinatorial freedom.     
    In contrast, the \emph{CAR algebra} is obtained from the direct sequence $\phi_r:\M_{2^r}\to\M_{2^{r+1}}$ defined by 
    $\phi_r(A) = A \oplus A$ \cite[Ex.~III.2.4]{davidson}.  Since
    there is only one choice of unital embedding possible at each level, there is nothing to analyze from a probabilistic perspective (moreover, $g_r = 2^r$ does not satisfy \eqref{eq:Rapid}).
    As another example, recall that a C*-algebra is \emph{uniformly hyperfinite} (UHF) if it is the increasing union of unital subalgebras isomorphic to full matrix algebras $\M_{n_r}$ \cite[Ex.~III.5.1]{davidson}.  Each UHF algebra is characterized by an increasing divisibility sequence.  Since
    there is only one standard unital embedding possible at each level, there are no statistical phenomena to study.  
    Similarly, the Effros--Shen algebras and other standard examples do not fit readily into our framework.   
\end{remark}

\subsection{A probability measure on $\eng$}\label{Subsection:ProbabilityMeasure}
The following probabilistic construction is standard; see \cite[Prop.~10.6.1]{CohnMeasure} or \cite[Thm.~A, p.~212]{HalmosMeasure}.
Each set $\E_r = \shom(\A_r,\A_{r+1})$ of standard unital *-homomorphisms is finite and nonempty, so we may endow it with the uniform probability measure $\P_r$, which is defined on the powerset $\mathscr{B}_r$ of $\E_r$.  Thus, we have a sequence of probability spaces $(\E_r, \mathscr{B}_r, \P_r)$.
Recall from \eqref{eq:EngDef} that $\eng = \prod_{r=1}^{\infty} \E_r$ and define
\begin{equation*}
    \mathscr{B} = \Big\{ \prod_{r=1}^{\infty} \mathcal{C}_r : \text{$\mathcal{C}_r \in \mathscr{B}_r$ for all $r$ and $\mathcal{C}_r = \E_r$ for all but finitely many $r$}\Big\}.
\end{equation*}
Then there is a unique probability measure $\P$ on $(\eng,\mathscr{B})$ such that
\begin{equation}\label{eq:ProbabilityProduct}
    \P(\prod_{r=1}^{\infty} \mathcal{C}_r) 
    = \prod_{r=1}^{\infty} \P_r(\mathcal{C}_r)
\end{equation}
for all $\prod_{r=1}^{\infty} \mathcal{C}_r \in \mathscr{B}$.  Let $\pi_r:\eng\to\E_r$
denote the $r$th coordinate function.  Then the random variables $\pi_r$ are independent and they have the distributions $\P_r$ on each measurable space $(\E_r, \mathscr{B}_r)$, respectively.

Thus, there is a natural probability measure on $\eng$ that satisfies \eqref{eq:ProbabilityProduct}.  This permits us to discuss properties that hold for almost every chain, for example.
We often appeal to the (first) Borel--Cantelli lemma:  if $E_1,E_2,\ldots$ are events such that $\sum_{r=1}^{\infty} \P(E_r) < \infty$, then the probability that infinitely many of the $E_r$ occur is $0$.  

\subsection{Asymptotic statistics}\label{Subsection:Stats}
The asymptotic statistical properties (as $r\to\infty$) of $\shom(\A_r,\A_{r+1})$ can be described using results about numerical semigroups \cite{factorization2, factorization5}. First we consider the cardinality of $\shom(\A_r,\A_{r+1})$.
Let $d=\gcd(n_1,n_2,\ldots,n_k)$.

\begin{proposition}\label{Proposition:HomTotal}
    $\displaystyle|\shom(\A_r,\A_{r+1})| \sim
    \frac{d^k \gamma_r^{k(k-1)}}{[(k-1)!]^k(n_1n_2\cdots n_k)}$ as $r \to \infty$.
\end{proposition}

\begin{proof} 
Recall that $d=\gcd(n_1,\dots,n_k)$. Use the notation \eqref{eq:RowsColumns} and observe that
\begin{align*}
|\shom(\A_r,\A_{r+1}) |
&= | \{ X \in \M_k(\Z_{\geq 0}) : X \vec{n} = \gamma_r \vec{n} \} | && \text{by \eqref{eq:Xngn}}\nonumber\\
&=\prod_{i=1}^k |\{ \vec{r}_i \in \Z_{\geq 0}^k : \vec{r}_i\cdot   \vec{n}  = \gamma_r n_i \} |  \nonumber&\\
&= \prod_{i=1}^k |\ZZ_S(\gamma_r n_i) | & \\
&\sim\prod_{i=1}^k \frac{ d(\gamma_r n_i)^{k-1}}{(k-1)!n_1n_2\dots n_k}  && \text{by \eqref{eq:Cardinal-with-d}}\nonumber\\
&= \frac{ d^k\gamma_r^{ k(k-1)} (n_1 n_2 \cdots n_k)^{k-1}}{[(k-1)!]^k (n_1 n_2 \cdots n_k)^k} \nonumber\\
&= \frac{ d^k\gamma_r^{ k(k-1)} }{[(k-1)!]^k (n_1 n_2 \cdots n_k) }.\label{eq:HowManyHoms}\qedhere
\end{align*}
\end{proof}

An important intermediate step in the previous proof is the equality
\begin{equation}\label{eq:HomZZ}
    |\shom(\A_r,\A_{r+1}) |
= \prod_{i=1}^k |\ZZ_S(\gamma_r n_i) |.
\end{equation}
This formula arises frequently in what follows.

The number of times $\phi_r\in \shom(\A_r,\A_{r+1})$ maps
$\M_{g_rn_j}$ into $\M_{g_{r+1}n_i}$ is $x_{ij}$, the $(i,j)$ entry of the associated partial-multiplicity matrix $X = [x_{ij}]\in \M_k(\Z_{\geq 0})$.  Each such matrix satisfies $X\n=\gamma_r\n$ by \eqref{eq:Xngn}.
For each $p \in \Z_{\geq 0}$, the $p$th moment of $x_{ij}$ is
\begin{equation*}
    \mu_p(x_{ij}) = \frac{1}{|\shom(\A_r,\A_{r+1})|}\sum_{X\n=\gamma_r\n} x_{ij}^p.
\end{equation*}
Do not confuse the $p$th power $x_{ij}^p$ of $x_{ij}$ with the notation $x_{ij}^{(r)}$ for the $(i,j)$ entry
of the matrix $X_r$; recall that the level $r$ is often suppressed.

We have the following expression for the asymptotic moments.

\begin{proposition}\label{Proposition:Moments}
$\displaystyle \mu_p(x_{ij}) \sim \bigg(\gamma_r\frac{n_i}{n_j}\bigg)^p \binom{p+k-1}{p}^{-1}$ as $r\to\infty$.
\end{proposition}

\begin{proof}
    First note that \eqref{eq:SplineSame} says that
    \begin{equation}\label{eq:LotsZeros}
    M\big(t;\underbrace{0,0,\ldots,0,\tfrac{d}{n_j},0,\ldots,0}_{\text{$\frac{d}{n_j}$ in $j$th position}}\big)
    =
    \begin{cases}
    (k-1)\big(\frac{n_j}{d}\big)^{k-1}\big( \frac{d}{n_j}-t\big)^{k-2}& \text{if $0\leq t \leq \frac{d}{n_j}$},\\
    0 &\text{otherwise}.
    \end{cases}
    \end{equation}
    Corollary \ref{Corollary:SemigroupWeak} with $f(t) = t^p$ yields
    \begin{align*}
        \frac{\sum_{X\n=\gamma_r\n}x_{ij}^p}{ | \shom(\A_r,\A_{r+1}) | }
        &=\frac{\prod_{h:h\neq i}|\ZZ_S(\gamma_rn_h)|}{\prod_{h=1}^k |\ZZ_S(\gamma_rn_h)|}
        \cdot
        \sum_{\vec{r}\in \ZZ_S(\gamma_rn_i)} (\vec{r}\cdot \vec{e}_j)^p
        &&\text{by \eqref{eq:HomZZ}}\\
        &=(\gamma_r n_i)^p \bigg( \frac{1}{|\ZZ_S(\gamma_rn_i)|}\cdot
        \sum_{\vec{r}\in \ZZ_S(\gamma_rn_i)} \Big(\frac{ \vec{r}\cdot \vec{e}_j }{ \gamma_r n_i}\Big)^p \bigg)\\
        &\sim (\gamma_r n_i)^p \int_{-\infty}^{\infty} (t/d)^p\, M\big(t;0,0,\ldots,0,\tfrac{d}{n_j},0,\ldots,0\big)  \,dt&&\text{by Cor~\ref{Corollary:SemigroupWeak}}\\
        &=\bigg(\frac{\gamma_r n_i}{d} \bigg)^p (k-1) \bigg(\frac{n_j}{d}\bigg)^{k-1}\int_{0}^{d/n_j} t^p\big(\tfrac{d}{n_j}-t\big)^{k-2}dt && \text{by \eqref{eq:LotsZeros}}\\
        &=\bigg(\frac{\gamma_r n_i}{d}\bigg)^p (k-1)\bigg(\frac{n_j}{d}\bigg)^{k-1}\bigg(\frac{d}{n_j}\bigg)^{k+p-1} \!\! \frac{(k-2)!p!}{(p+k-1)!}\\
        &=\bigg(\gamma_r\frac{n_i}{n_j}\bigg)^p  \frac{(k-1)!p!}{(p+k-1)!}\\
        &= \bigg(\gamma_r\frac{n_i}{n_j}\bigg)^p \binom{k+p-1}{p}^{-1}.\qedhere
    \end{align*}
\end{proof}

\begin{corollary}\label{Corollary:Mean}
The mean number of times $\M_{g_rn_j}$ is mapped into $\M_{g_{r+1}n_i}$ satisfies
\begin{equation*}
    \mu_1(x_{ij}) \sim \frac{\gamma_r}{k}\cdot \frac{n_i}{n_j}.
\end{equation*}
\end{corollary}

To go further, we need the next lemma.

\begin{lemma}\label{Lemma:AsympototicSum}
    If $f_i(r)\sim g_i(r)$ as $r\to \infty$ for $1 \leq i\leq k$ and if for each $(i,j)$, the quotient $g_i(r)/g_{j}(r)$ is nonzero and constant with respect to $r$, then $\sum_{i=1}^k f_i(r)\sim\sum_{i=1}^k g_i(r)$.
\end{lemma}

\begin{proof}
We proceed by induction on $k$.  The base case $k=1$ is immediate.
Suppose that $\sum_{i=1}^{k-1}f_i(r)\sim\sum_{i=1}^{k-1}g_i(r)$.
For clarity, we drop the argument $r$ in what follows.  The hypotheses ensure that $(\sum_{i=1}^{k-1}g_i)/g_k$
is constant, call it $c$.  Then
\begin{align*}
    \frac{\sum_{i=1}^k f_i}{\sum_{i=1}^k g_i}
    &=\bigg[\frac{\sum_{i=1}^{k-1}f_i}{\sum_{i=1}^{k-1}g_i}\bigg(\frac{1}{1+ g_k/(\sum_{i=1}^{k-1}g_i)}\bigg) + \frac{f_k}{g_k}\bigg(\frac{1}{1+ (\sum_{i=1}^{k-1}g_i)/g_k}\bigg)\bigg]\\
    &=\left(\frac{1}{1+\frac{1}{c}}\right) \frac{\sum_{i=1}^{k-1}f_i}{\sum_{i=1}^{k-1}g_i}+\left(\frac{1}{1+c}\right) \frac{f_k}{g_k}
    \sim\frac{1}{1+\frac{1}{c}}+\frac{1}{1+c}=1.\qedhere
\end{align*}
\end{proof}

\begin{proposition}\label{prop:avg-nb-in-image}
    As $r\to\infty$, the mean number of times $\mathcal{M}_{g_rn_j}$ appears in $\phi_r(\A_r)$ is
    \begin{align}\label{eq:avg-nb-in-image}
        \sum_{i=1}^k \mu_1(x_{ij}) \sim
        \frac{\gamma_r}{n_j} \bigg(\frac{1}{k}\sum_{i=1}^kn_i \bigg).
    \end{align}
\end{proposition}

\begin{proof}
    Apply Corollary~\ref{Corollary:Mean} and Lemma~\ref{Lemma:AsympototicSum}.
\end{proof}

We next compute the asymptotic standard deviation
\begin{equation*}
\sigma(x_{ij}) =\sqrt{\mu_2(x_{ij})-\mu_1^2(x_{ij})}.
\end{equation*}
The previous lemma is required since $\sim$ does not respect addition in general.

\begin{proposition}\label{Proposition:StDev}
    $\displaystyle \sigma(x_{ij}) \sim \frac{1}{k}\sqrt{ \frac{k-1}{k+1} }\bigg(\gamma_r\frac{n_i}{n_j}\bigg)$ as $r\to\infty$.
\end{proposition}

\begin{proof}
Lemma \ref{Lemma:AsympototicSum} ensures that the variance of $x_{ij}$ satisfies
\begin{align*}
    \mu_2(x_{ij})-\mu_1^2(x_{ij}) 
    &\sim \frac{2}{k(k+1)}\bigg(\gamma_r\frac{n_i}{n_j}\bigg)^2  - \frac{1}{k^2} \bigg(\gamma_r\frac{n_i}{n_j}\bigg)^2 \\
    &= \bigg( \frac{2}{k(k+1)} - \frac{1}{k^2}\bigg) \bigg(\gamma_r\frac{n_i}{n_j}\bigg)^2\\
    &= \frac{1}{k^2}\bigg( \frac{k-1}{k+1} \bigg)\bigg(\gamma_r\frac{n_i}{n_j}\bigg)^2
\end{align*}
as $r \to \infty$.  Take square roots and obtain the desired result.
\end{proof}

\begin{proposition}\label{Proposition:Median}   
$\displaystyle\operatorname{median} x_{ij} \sim (1-2^{\frac{1}{1-k}})\bigg(\gamma_r\frac{n_i}{n_j}\bigg)$ as $r\to\infty$.
\end{proposition}

\begin{proof}
    Let $M(t;0,0,\ldots,0,\tfrac{d}{n_j},0,\ldots,0)$ denote the spline \eqref{eq:LotsZeros}
    and observe that
    \begin{equation*}
        \int_{-\infty}^z M(t;0,0,\ldots,0,\tfrac{d}{n_j},0,\ldots,0)\,dt = 1-\left(1-\frac{n_j z}{d}\right)^{k-1}
    \end{equation*}
    for $0< z < d/n_j$.  Next solve $1-(1-\frac{n_jz}{d})^{k-1} = \frac{1}{2}$ and obtain
    $z = d(1-2^{\frac{1}{1-k}}) /n_j$,
    the median of the spline.  In Corollary \ref{Corollary:SemigroupWeak}, use $I = (-\infty,z/d]$, 
    $f = 1$ (constant function), $\vec{m} = \vec{e}_j$, 
    and $n = \gamma_r n_i$, to obtain 
    \begin{align*}
        \frac{1}{|\ZZ_S(n)|} \sum_{ \substack{ \vec{x} \in \ZZ_S(n) \\ (\vec{e}_j\cdot \vec{x})/n \in I }}  
        \!\!\!\!\!\!\!\!1\,\,
        &\sim \int_{-\infty}^z M(t; 0, 0, \ldots, 0,\tfrac{d}{n_j},0,\ldots,0 )\,dt = \frac{1}{2}
    \end{align*}
    as $r \to \infty$.
    The condition $(\vec{e}_j \cdot \vec{x})/n \in (-\infty, z/d]$ is equivalent to
    $x_{ij} \in (-\infty, nz/d]$.  The equation that defines $x_{ij}$ in \eqref{eq:ADN}
    has $n = \gamma_r n_i$.  Therefore, 
    \begin{equation*}
        \displaystyle\operatorname{median} x_{ij}
        \sim \frac{nz}{d} = \gamma_r n_i (1-2^{\frac{1}{1-k}}) /n_j = (1-2^{\frac{1}{1-k}})\bigg(\gamma_r\frac{n_i}{n_j}\bigg). \qedhere
    \end{equation*}
\end{proof}

\begin{proposition}\label{prop:number-of-matrix-alg-in-image}
    Let $\phi_r$ range over $\hom(\A_r,\A_{r+1})$.  Then as $r \to \infty$,
    the mean number of matrix algebras in $\phi_r(\A_r)\subseteq \A_{r+1}$ is asymptotically equivalent to
    \begin{equation}\label{eq:algs-in-image}
        \frac{\gamma_r}{k}(n_1+n_2+\cdots+n_k)\left(\frac{1}{n_1}+\frac{1}{n_2} + \cdots +\frac{1}{n_k}\right)
        = 
        \g_r k\cdot \frac{A(n_1,n_2,\ldots,n_k)}{H(n_1,n_2,\ldots,n_k)},
    \end{equation}
    in which $A(n_1,\ldots,n_k)$ and $H(n_1,\ldots,n_k)$ denote the arithmetic and harmonic means of the generators, respectively.
\end{proposition}

\begin{proof}
    Sum \eqref{eq:avg-nb-in-image} over $j$ and appeal to Lemma \ref{Lemma:AsympototicSum}.
\end{proof}

We can generalize this result by assigning weights to the generators. Given $\m=(m_1,m_2,\ldots,m_k)\in \R^k$, 
in which $\m$ and $\vec{n}$ are linearly independent, each time a copy of $\M_{n_j}$ appears in $\phi_r(\M_{n_j})$,
count it as appearing $m_j$ times. Note that $\m$ need not consist of integral or even non-negative entries.

\begin{proposition}\label{prop:image-with-weightings}
    Let $\m=(m_1,\ldots,m_k)\in \R^k$ with $\m,\vec{n}$ linearly independent.  The weighted number of copies of matrix algebras in $\phi_r(\A_r)\subseteq \A_{r+1}$ is asymptotically equivalent to
    \begin{equation}\label{eq:weighted-algs-in-image}
        \frac{\gamma_r}{k}(n_1+\cdots+n_k)\left(\frac{m_1}{n_1}+\cdots +\frac{m_k}{n_k}\right).
    \end{equation}
\end{proposition}

\begin{proof}
        Weight the result from \eqref{eq:avg-nb-in-image} by $m_j$ and apply 
        Lemma \ref{Lemma:AsympototicSum}.
\end{proof}

\subsection{Examples}\label{Subsection:Examples}
In this subsection, we illustrate the previous material with several illustrative examples.
Since asymptotic statistics about any given $\eng$ are computable using standard software, 
this sheds light on the typical algebras in $\eng$.

\begin{example}\label{ex:6-9-20}
    Consider $\eng$, in which $\n = (6,9,20)$ and $g_r=r!$. Figure~\ref{Figure:6-9-20} and 
    Table~\ref{tab:6-9-20} concern the values of the $(1,1)$ entry $x_{11}^{(499)}$ as $X_{499}$ runs over all partial multiplicity matrices corresponding to standard unital $*$-homomorphisms between 
    $$\A_{499}=\M_{499!\cdot 6}\oplus\M_{499! \cdot  9}\oplus\M_{499! \cdot 20}$$ and 
    $$\A_{500}=\M_{500!\cdot 6}\oplus\M_{500!\cdot 9}\oplus\M_{500!\cdot 20}.$$
    Similarly, we compute $\P\big(200\leq x_{2,1}^{(499)}\leq 500\big)=42.71\%$ directly, which
    agrees with Proposition~\ref{Proposition:EdgeWeights}, or \eqref{eq:LocalLimit}.
\begin{figure}[h!]
\centering
  \includegraphics[width=0.5\textwidth]{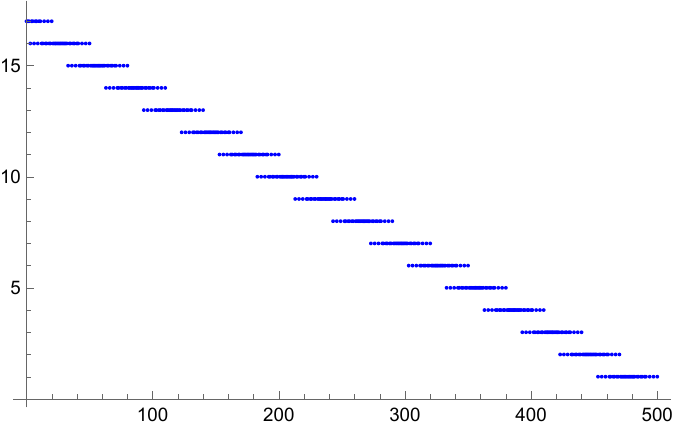}

  \caption{Histogram for $x_{11}^{(499)}$ from Example~\ref{ex:6-9-20}, in which $\n = (6,9,20)$ and $g_r=r!$. The vertical axis represents multiplicities as $x_{11}^{(499)}$ runs over all partial multiplicity matrices corresponding to unital $*$-homomorphisms between $\A_{499}$ and $\A_{500}$.}
\label{Figure:6-9-20}
\end{figure}

\begin{table}[h!]
    \centering
    \begin{tabular}{c|c|c}
      Statistic& Actual & Predicted \\
       \hline
      Mean & $168.49$ & $166.67$ (by \eqref{eq:avg-nb-in-image})\\
      Median & $148$ & $146.45$ (by \ref{Proposition:Median})\\
      Mode & \{0,1,2,4,5,7,8,10,11,14,17,20\} &   $0$ (by \ref{Proposition:EdgeWeights})\\
      Standard Deviation & $119.50$ &  $117.85$ (by \ref{Corollary:Mean})
    \end{tabular}
    \caption{Computed statistics of $x_{11}^{(499)}$ corresponding to Figure~\ref{Figure:6-9-20} and Example~\ref{ex:6-9-20}}\label{tab:6-9-20}
\end{table}

\end{example}

\begin{example}\label{ex:7-13-14-15}
The distribution of values of $x_{ij}^{(r)}$ is not always triangular. Consider $\eng$ with $\vec{n}=(7,13,14,15)$ and $g_r=r!$. In Figure~\ref{fig:7-13-14-15}, we exhibit statistics on $x_{(4j)}^{(100)}$ for every generator $n_j\in\n$, and in Table~\ref{tab:7-13-14-15} statistics are provided for $x_{44}^{(100)}$.

\begin{figure}[htbp]
\centering
\begin{subfigure}{0.475\textwidth}
  \centering
  \includegraphics[width=\textwidth]{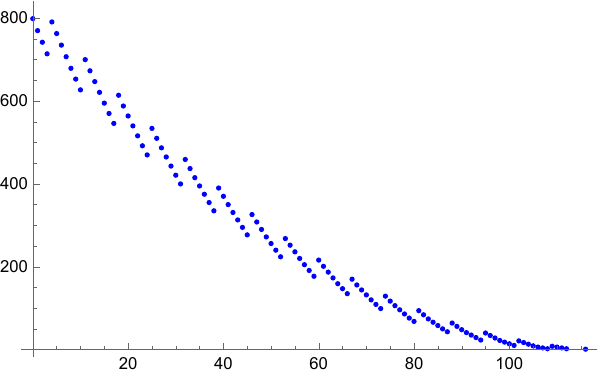}
  \caption{Distribution of values of $x_{41}^{(100)}$.}
\end{subfigure}
\hfill
\begin{subfigure}{0.475\textwidth}
  \centering
  \includegraphics[width=\textwidth]{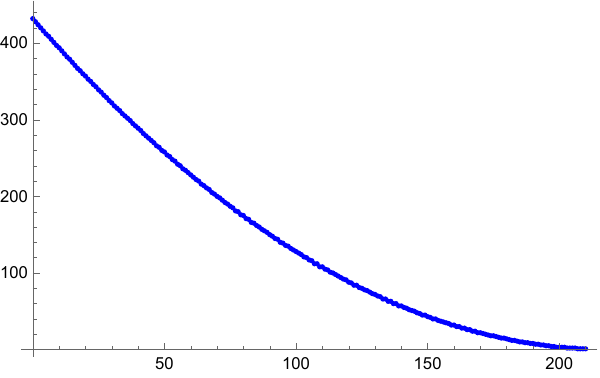}
  \caption{Distribution of values of $x_{42}^{(100)}$.}
\end{subfigure}
\\
\begin{subfigure}{0.475\textwidth}
  \centering
  \includegraphics[width=\textwidth]{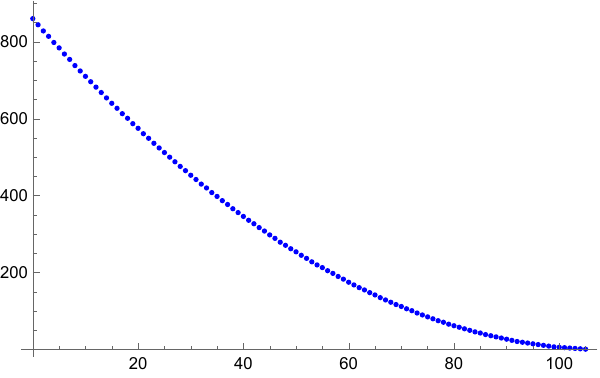}
  \caption{Distribution of values of $x_{43}^{(100)}$.}
\end{subfigure}
\hfill
\begin{subfigure}{0.475\textwidth}
  \centering
  \includegraphics[width=\textwidth]{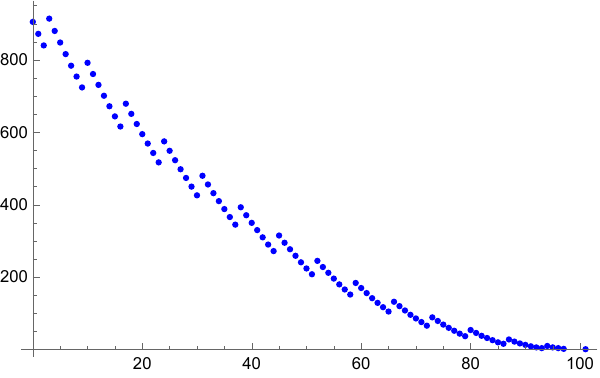}
  \caption{Distribution of values of $x_{44}^{(100)}$.}
  \label{subfig:7-13-14-15}
\end{subfigure}
\hfill
\caption{Distribution of values of $x_{4j}^{(499)}$ for $j=1,2,3,4$ as in Example~\ref{ex:7-13-14-15}. For this family $\eng$, $\A_{100}=\M_{100!\cdot 7}\oplus\M_{100!\cdot 13}\oplus\M_{100!\cdot 14}\oplus\M_{100!\cdot 15}$ and $\A_{101}=\M_{101!\cdot 7}\oplus\M_{101!\cdot 13}\oplus\M_{101!\cdot 14}\oplus\M_{101!\cdot 15}$. We compute values of $x_{4j}^{(499)}$ for all possible such $X_{499}$. Thus, these distributions show the number of times $\M_{100!\cdot j}$, the $j$th block of $\A_{100}$, appears in $\M_{101!4}$, the $4$th block of $\A_{101}$.}
\label{fig:7-13-14-15}
\end{figure}

\begin{table}[h!]
    \centering
    \begin{tabular}{c|c|c}
        Statistic & Actual & Predicted
        \\
        \hline
         Mean &$25.18$&  $25.25$ (by \ref{Corollary:Mean})\\
          Median & $21$& $20.84$ (by \ref{Proposition:Median})\\
          Mode & $3$ & $0$ (by \ref{Proposition:EdgeWeights})\\
         Std Deviation& $19.86$ & $19.56$ (by \ref{Proposition:StDev})
    \end{tabular}
    \caption{Comparing statistics of $x_{11}^{(499)}$ corresponding to Subfigure~\ref{subfig:7-13-14-15} and Example~\ref{ex:7-13-14-15} and predicted values.}
    \label{tab:7-13-14-15}
\end{table}

\end{example}

\begin{example}\label{ex:gcd3}
    Consider $\eng$ with $\n=(3,12,18)$ and $g_r=2^{(r-1)^2}$, and $f(x)=\sin(x)$. In this case $\gcd(3,12,18)=3$.  Figure~\ref{Figure:gcd3} exhibits the distribution of $\sin(x_{i2}^{(5)})$ for each generator $n_i$.

    \begin{figure}[h!]
\centering
\begin{subfigure}{0.475\textwidth}
  \centering
  \includegraphics[width=\textwidth]{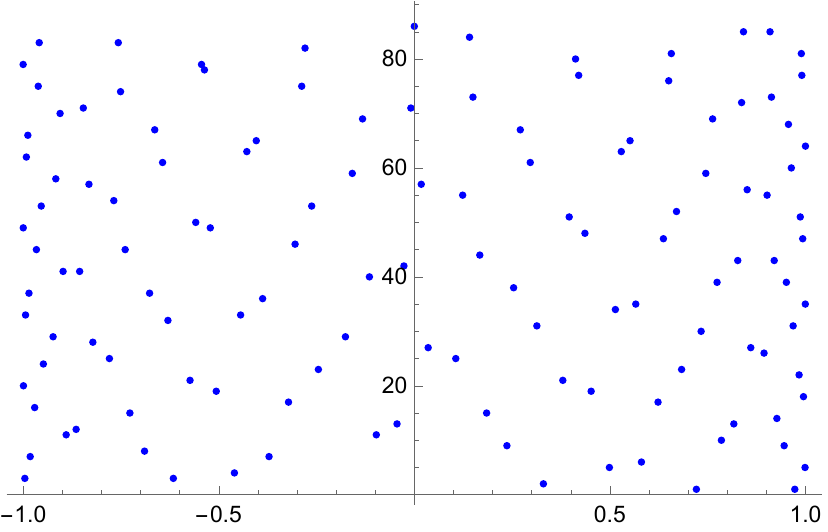}
  \caption{Distribution of values of $\sin(x_{12}^{(5)})$.}
  \label{subfig:gcd3}
\end{subfigure}
\hfill
\begin{subfigure}{0.475\textwidth}
  \centering
  \includegraphics[width=\textwidth]{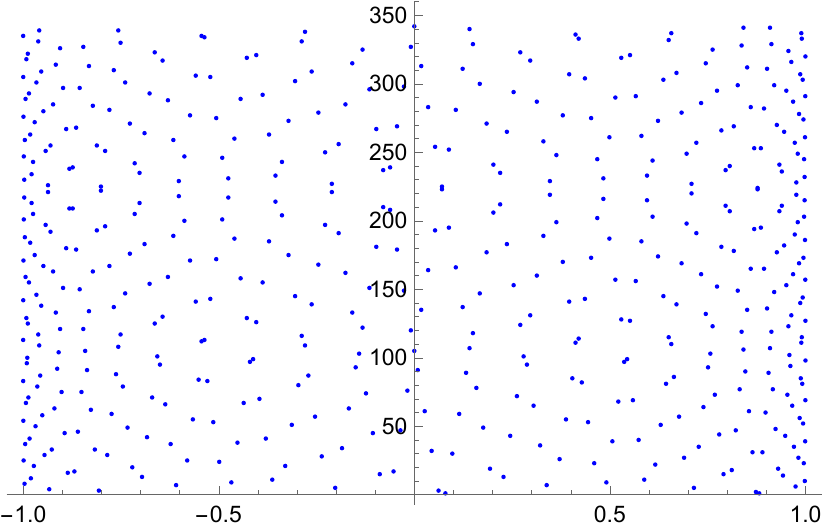}
  \caption{Distribution of values of $\sin(x_{22}^{(5)})$.}
\end{subfigure}
\\
\begin{subfigure}{0.475\textwidth}
  \centering
  \includegraphics[width=\textwidth]{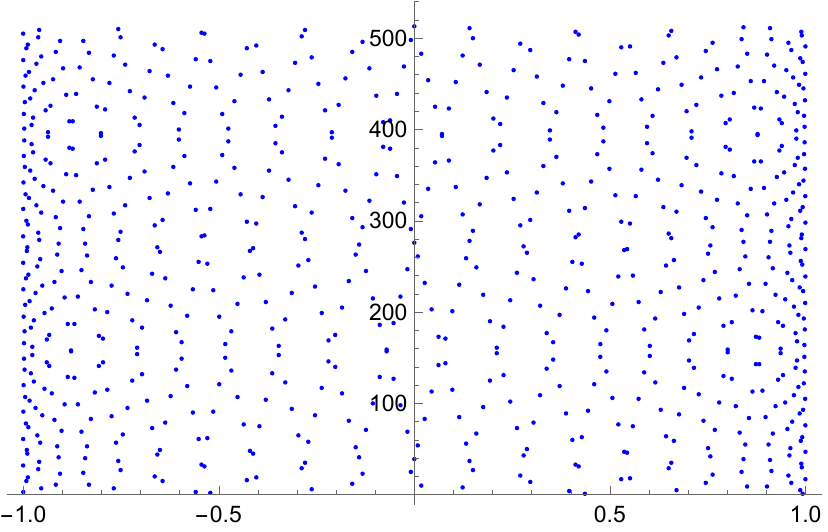}
  
  \caption{Distribution of values of $\sin(x_{32}^{(5)})$.}
\end{subfigure}
\hfill

\caption{Distribution of $\sin(x_{i2}^{(5)})$ for every generator $n_i$ as in Example~\ref{ex:gcd3}, with the $x$-axis representing the value of $\sin(x_{i2}^{(5)})$, and the $y$-axis represents the multiplicity of this value. For the family $\eng$, $\A_{6}=\M_{2^{16}\cdot 3}\oplus\M_{2^{16}\cdot 12}\oplus\M_{2^{16}\cdot 18}$ and $\A_{7}=\M_{2^{25}\cdot 3}\oplus\M_{2^{25}\cdot 12}\oplus\M_{2^{25}\cdot 18}$. Recall $x_{i2}^{(5)}$ corresponds to the $(i,2)$ entry of the partial multiplicity matrix $X_{5}$ corresponding to some $\phi\in \hom(\A_{5},\A_{6})$. We compute values of $\sin(x_{i2}^{(5)})$ for all possible such $X_{5}$. Thus, these distributions show the {sine of the} number of times $\M_{2^{16}12}$, the second block of $\A_{5}$, appears in the $\M_{2^{25}i}$, the $i$th block of $\A_{6}$. }
\label{Figure:gcd3}
\end{figure}

\end{example}

\section{Injectivity}\label{Section:Injective}

In this section, we address injectivity:  how often is $\phi_r:\A_r\to\A_{r+1}$ injective?
This culminates in Theorem \ref{Theorem:InjHoms}, which asserts that with probability $1$, a randomly selected chain in $\eng$ has all but finitely many homomorphisms injective.  This justifies our treatment of $\eng$ as a collection of AF algebras.

\begin{definition}
    The generator $n_i$ \emph{vanishes at the $r$th level} of a chain $\A \in \eng$ if, in the Bratteli diagram $B(\A)$, the vertex corresponding to $n_i$ at level $r$ is adjacent to no vertices at level $r+1$.
We sometimes say that a generator \emph{vanishes with respect to $\phi_r$} as an indirect way of specifying $\A$.
\end{definition}

\begin{example}
    Fix $\vec{n}=(n_1,n_2,n_3)$ with $\gcd(n_1,n_2,n_3)=1$.
    The following figure shows the first three levels of the Bratteli diagram associated to an AF algebra in $\eng$. 
    The multiplicities of the embeddings are unspecified for clarity.
    In this case, $n_2$ vanishes at the second level (red) and $n_3$ vanishes at the third level (blue).
    \begin{equation*}\begin{tikzcd}[cramped,sep=scriptsize]
	& {\mathcal{A}_1} & {\mathcal{A}_2} & {\mathcal{A}_3} & {\mathcal{A}_4} \\[-10pt]
	{n_1} & \bullet & \bullet & \bullet & \bullet \\
	{n_2} & \bullet & \color{red}\bullet & \bullet & \bullet \\
	{n_3} & \bullet & \bullet & \color{blue}\bullet & \bullet
	\arrow[from=2-2, to=2-3]
	\arrow[from=2-2, to=3-3]
	\arrow[from=2-3, to=2-4]
	\arrow[from=2-3, to=3-4]
	\arrow[from=2-3, to=4-4]
	\arrow[from=2-4, to=3-5]
	\arrow[from=3-2, to=2-3]
	\arrow[from=3-2, to=3-3]
	\arrow[from=3-2, to=4-3]
	\arrow[from=3-4, to=2-5]
	\arrow[from=3-4, to=3-5]
	\arrow[from=3-4, to=4-5]
	\arrow[from=4-2, to=4-3]
	\arrow[from=4-3, to=4-4]
\end{tikzcd}
\end{equation*}
Note that $\phi_r:\A_r\to\A_{r+1}$ is injective if and only if no generators vanish at the $r$th level. Thus, $\phi_1$ is injective, but $\phi_2$ and $\phi_3$ are not.
\end{example}

\begin{proposition}\label{Proposition:AllVanish}
    As $r\to\infty$,
    the probability that every generator indexed by $M\subsetneq \{1,2,\ldots,k\}$ vanishes at level $r$ is asymptotically equivalent to
    \begin{equation*}
    \left(\frac{(k-1)!}{(k-|M|-1)!}\right)^k\frac{\prod_{m\in M}n_m^k}{(n_1 n_2 \cdots n_k)^{|M|}\cdot \gamma_r^{k|M|}}.
    \end{equation*}
\end{proposition}

\begin{proof}
Let $M \subsetneq \{1,2,\ldots,k\}$. The number $N$ of $\phi \in \shom(\A_r,\A_{r+1})$
that place no copies of $\M_{g_rn_m}$ in $\M_{g_{r+1}n_m}$ for each $m \in M$ is the number of
$X = [ \vec{c}_1~\vec{c}_2~\cdots~ \vec{c}_k] \in \M_k(\Z_{\geq 0})$
such that $X \vec{n} = \gamma_r\vec{n}$ and $\vec{c}_m = \vec{0}$ for all $m \in M$.
Let $S'$ be the semigroup generated by $\{n_1,n_2,\ldots,n_k\}\setminus\{n_m :  m\in M\}$. Then
\eqref{eq:Cardinal-with-d} ensures that
\begin{align}
    N
    &=|\{X\in \M_k(\Z_{\geq 0}) :  \text{$X\vec{n}=\gamma_r\vec{n}$ and $\vec{c}_m = \vec{0}$ for all $m\in M$}\} | \nonumber\\
    &=|\{X\in \M_k(\Z_{\geq 0}) :  \text{$\vec{r}_i\cdot \vec{n} = \gamma_r n_i$ for all $1\leq i \leq k$ and $\vec{c}_m = \vec{0}$ for all $m\in M$}\} | \nonumber\\
    &=\prod_{i=1}^k|\{ \vec{r}_i\in \Z_{\geq 0}^k : \text{$\vec{r}_i\cdot \vec{n} = \gamma_r n_i$ and $\vec{r}_i\cdot  \vec{e}_m = 0$ for all $1\leq i \leq k$ and $m\in M$}\}| \nonumber \\
    &=\prod_{i=1}^k|Z_{S'}(\gamma_r n_i)| \nonumber \\
    &=\frac{d^k\gamma_r^{k(k-|M|-1)}\prod_{m\in M}n_m^k}{[(k-|M|-1)!]^k(n_1 n_2 \cdots n_k)^{|M|+1}}. \label{eq:NME}
\end{align}
Divide by $|\shom(\A_r,\A_{r+1})|$, use Proposition \ref{Proposition:HomTotal}, and obtain the desired result:
\begin{align*}
    \frac{N}{| \shom(\A,\A_{r+1}) |}
    &\sim \frac{d^k\gamma_r^{k(k-|M|-1)}\prod_{m\in M}n_m^k}{[(k-|M|-1)!]^k(n_1 n_2 \cdots n_k)^{|M|+1}}
    \cdot \frac{[(k-1)!]^k(n_1 n_2 \cdots n_k)}{d^k\gamma_r^{k(k-1)}}\\
    &=\left(\frac{(k-1)!}{(k-|M|-1)!}\right)^k\frac{\prod_{m\in M}n_m^k}{\gamma_r^{k|M|}(n_1 n_2 \cdots n_k)^{|M|}}.
    \qedhere
\end{align*}
\end{proof}

\begin{remark}
If $|M|=0$, the expression in the previous proposition is $1$. This is expected:
the probability that zero or more generators vanished is obviously $1$. 
\end{remark}

\begin{corollary}\label{cor:proba-h-vanish}
Let $1 \leq h \leq k$.
    The probability that $h$ or more generators vanish at the $r$th level is asymptotically upper bounded by
    \begin{equation}\label{eq:proba-h-vanish}
        \frac{(h!)^k}{\gamma_r^{kh}} \cdot \binom{k-1}{h}^k \cdot \frac{e_h(n_1^k,n_2^k,\ldots,n_k^k)}{(n_1 n_2\cdots n_k)^{h}},
    \end{equation}
     in which $e_h$ denotes the elementary symmetric polynomial of degree $h$.
\end{corollary}

\begin{proof}
For each $M \subsetneq \{1,2,\ldots,k\}$, the asymptotic number of $\phi \in \shom(\A_r,\A_{r+1})$
that place no copies of $\M_{g_rn_m}$ in $\M_{g_{r+1}n_m}$ for each $m \in M$ is given by \eqref{eq:NME}.
Use Lemma \ref{Lemma:AsympototicSum}, sum \eqref{eq:NME} over all $M\subsetneq\{1,2,\ldots,k\}$ of size $h$, and obtain
\begin{equation*}
\frac{d^k\gamma_r^{k(k-h-1)}}{[(k-h-1)!]^k(n_1\cdots n_k)^{h+1}}\sum_{\substack{M\subseteq\{1,2,\ldots,k\}\\|M|=h}}\prod_{m\in M}n_m^k
= \frac{d^k\gamma_r^{k(k-h-1)} e_h(n_1^k,n_2^k,\ldots,n_k^k) }{[(k-h-1)!]^k(n_1\cdots n_k)^{h+1}} .
\end{equation*}
This is an asympototic upper bound on the number of standard homomorphisms at the $r$th level that annihilate $h$ or more generators.  Divide this expression by $|\shom(\A_r,\A_{r+1})|$, use Proposition \eqref{Proposition:HomTotal}, and obtain the desired result.
\end{proof}

The previous results culminate in following theorem.

\begin{theorem}\label{Theorem:InjHoms}
Those chains such that all but finitely many of the $\phi_r$ homomorphisms are injective form a set of full measure in $\eng$.
That is, with probability $1$, a randomly selected chain in $\eng$ has all but finitely many homomorphisms injective.
\end{theorem}

\begin{proof}
The hypothesis \eqref{eq:Rapid} ensures that $\sum_{r=1}^{\infty} \gamma_r^{-k} < \infty$ for each $k\geq 2$.
Let $E_r$ denote the event that an element of $\shom(\A_r,\A_{r+1})$ is noninjective.
This occurs if and only if at least one generator vanishes at level $r$ with respect to $\phi_r$. Corollary \ref{cor:proba-h-vanish} with $h=1$ ensures that
\begin{equation*}
\P(E_r)\lesssim
\frac{1}{\gamma_r^{k}} \cdot \frac{(k-1)^k\sum_{i=1}^k n_i^k}{(n_1 n_2 \cdots n_k)} = \frac{c}{\gamma_r^k}
\end{equation*}
as $r\to\infty$, so $\sum_{r=1}^{\infty} \P(E_r)$ converges.
The Borel--Cantelli lemma ensures that
the probability that a randomly selected chain contains infinitely many noninjective homomorphisms is zero.
\end{proof}

Recall that telescoping, the second Bratelli-diagram move, ensures that removing any finite number of initial levels does not change the isomorphism class of the resulting AF algebra.  As a consequence of the previous theorem, we may think of the elements of $\eng$ as AF algebras, and hence we may speak of randomly selected algebras drawn from the ensemble $\eng$.  

\section{Simplicity}\label{Section:Simplicity}

Each chain $\A\in \eng$ is associated to a unique Bratteli diagram, denoted $B(\A)$, whose edges at the $r$th level are determined by the associated partial multiplicity matrix $X_r$. 
A \emph{directed subdiagram} $D$ of $B(\A)$ is a subdiagram such that if $(r,i)\in D$ and there is an edge from $(r,i)$ to $(r+1,j)$ then $(r+1,j)\in D$. A \emph{hereditary subdiagram} $D$ is one where for each vertex $(r,i)$, if $(r+1,j)\in D$ for all $j$ such that there 
is an edge from $(r,i)$ to $(r+1,j)$, then $(r,i)\in D$. 
We say there is a \emph{path} from $(r,i)$ to $(r+h,j)$ if there exists a sequence of adjacent vertices
\begin{equation*}
(r,i)=(r,{i_0})\to (r+1,{i_1})\to\cdots\to (r+h,{i_h})=(r+h,j).
\end{equation*}
Two paths containing the same vertices are considered distinct if their edges have different weights.

Since Theorem \ref{Theorem:InjHoms} ensures that almost every chain in $\eng$ gives rise to an AF algebra, we are in essence studying random AF algebras drawn from the ensemble $\eng$.
The ideals of each $\A \in \eng$ are in bijection with the directed, hereditary sets of $B(\A)$ \cite[Thm~III.4.2]{davidson}.
Therefore, $\A$ is simple if and only if, for each vertex $(r,i)$, there exists an $R>r$ such that for all $1\leq j\leq k$, there is a path from $(r,i)$ to $(R,j)$. This is equivalent to saying that for any $r$, if one removes the first $r$ levels of $B(\A)$, the resulting Bratteli diagram is connected.

\begin{example}
In the diagram below, we have suppressed the weights of the edges and the dimensions of the vertices. Assume all levels not shown are the same as the last two levels (modulo weights). Then the associated algebra is simple.
\tikzcdset{row sep/normal=0.4cm}
\begin{equation*}\begin{tikzcd}[cramped]
	\bullet & \bullet & \bullet & \bullet & \bullet & \bullet & \bullet & \bullet & \cdots \\
	\bullet & \bullet & \bullet & \bullet & \bullet & \bullet & \bullet & \bullet & \cdots \\
	\bullet & \bullet & \bullet & \bullet & \bullet & \bullet & \bullet & \bullet & \cdots
	\arrow[from=1-1, to=1-2]
	\arrow[from=1-2, to=1-3]
	\arrow[from=1-3, to=1-4]
	\arrow[from=1-4, to=1-5]
	\arrow[from=1-4, to=2-5]
	\arrow[from=1-5, to=2-6]
	\arrow[from=1-6, to=1-7]
	\arrow[from=1-6, to=2-7]
	\arrow[from=1-7, to=1-8]
	\arrow[from=1-7, to=2-8]
	\arrow[from=1-8, to=1-9]
	\arrow[from=1-8, to=2-9]
	\arrow[from=2-1, to=1-2]
	\arrow[from=2-1, to=2-2]
	\arrow[from=2-2, to=2-3]
	\arrow[from=2-3, to=2-4]
	\arrow[from=2-3, to=3-4]
	\arrow[from=2-4, to=2-5]
	\arrow[from=2-4, to=3-5]
	\arrow[from=2-5, to=1-6]
	\arrow[from=2-6, to=2-7]
	\arrow[from=2-6, to=3-7]
	\arrow[from=2-7, to=1-8]
	\arrow[from=2-7, to=2-8]
	\arrow[from=2-7, to=3-8]
	\arrow[from=2-8, to=1-9]
	\arrow[from=2-8, to=2-9]
	\arrow[from=2-8, to=3-9]
	\arrow[from=3-1, to=3-2]
	\arrow[from=3-2, to=1-3]
	\arrow[from=3-2, to=3-3]
	\arrow[from=3-3, to=3-4]
	\arrow[from=3-4, to=2-5]
	\arrow[from=3-5, to=2-6]
	\arrow[from=3-5, to=3-6]
	\arrow[from=3-6, to=3-7]
	\arrow[from=3-7, to=2-8]
	\arrow[from=3-7, to=3-8]
	\arrow[from=3-8, to=2-9]
        \arrow[from=3-8, to=3-9]
\end{tikzcd}\end{equation*}
\end{example}


\begin{proposition}\label{Proposition:EdgeWeights}
    For each $w \in \Z_{\geq 0}$,
    \begin{equation*}
        \P\Big(x_{ij}^{(r)} = w \Big) \sim
        \frac{(k-1)n_j}{\gamma_r n_i}\bigg(1 - \frac{wn_j}{\gamma_r n_i} \bigg)^{k-2}
    \end{equation*}
    as $r \to \infty$.
    In particular, $\operatorname{mode} x_{ij}^{(r)} \sim 0$.
\end{proposition}

\begin{proof}
    Fix $r \geq 1$.  We must count the number of partial multiplicity matrices $X \in \M_k(\Z)$, which correspond to
    $\phi \in \shom(\A_r,\A_{r+1})$, with $x_{ij}=w$.  Suppose
    that $X$ is such a matrix and let $X$ be partitioned according to its columns and rows as in \eqref{eq:RowsColumns}, with columns $\vec{c}_j$ and rows $\vec{r}_i\cdot $ for $1 \leq i,j \leq k$. 
    Suppose that $x_{ij}=w$ and let
    $S'=\semigroup{ n_1,\ldots, n_{j-1}, n_{j+1},\ldots,n_k}$. Then
    \begin{align*}
        &|\{X\in \M_k(\Z_{\geq 0}) : X\vec{n}=\gamma_r\vec{n},\, x_{ij}=w\} |\\
        &\qquad=|\{\vec{r}_i\in \Z_{\geq 0}^k: \vec{r}_i\cdot \vec{e}_j=w,\, \vec{r}_i\cdot \vec{n}=n_i\gamma_r\}|\cdot \prod_{\substack{h=1\\h\neq i}}^k|\{\vec{r}\in \Z_{\geq 0}^k: \vec{r}\cdot\vec{n}=\gamma_r n_h\}|\\
        &\qquad=|\ZZ_{S'}(\gamma_rn_i-wn_j)|\cdot \prod_{\substack{h=1\\h\neq i}}^k |\ZZ_S(\gamma_rn_h)|.
    \end{align*}
    Then \eqref{eq:Cardinal-with-d}, \eqref{eq:Z-gcd-change}, and \eqref{eq:HomZZ} ensure that
    \begin{align*}
        &\frac{|\{X\in \M_k(\Z_{\geq0}): X\vec{n}=\gamma_r\vec{n},\, x_{ij}=w\} |}{|\shom(\A_r,\A_{r+1})|} 
         = \frac{|\ZZ_{S'}(\gamma_rn_i-wn_j)|}{|\ZZ_S(\gamma_rn_i)|} \\
        &\qquad \sim \frac{d(\gamma_rn_i-wn_j)^{k-2}}{(k-2)!(n_1\cdots n_{j-1}n_{j+1}\cdots n_k)} \cdot
        \frac{(k-1)!(n_1n_2\cdots n_k)}{d(\gamma_r n_i)^{k-1}}&&\text{by \eqref{eq:Cardinal-with-d}}\\
        &\qquad = \frac{(k-1) (\gamma_r n_i)^{k-2}\big(1 - \frac{wn_j}{\gamma_r n_i} \big)^{k-2} n_j}{(\gamma_r n_i)^{k-1}}\\
        &\qquad = \frac{(k-1)n_j}{\gamma_r n_i}\bigg(1 - \frac{wn_j}{\gamma_r n_i} \bigg)^{k-2} . \qedhere
    \end{align*}
\end{proof}

\begin{remark}
    Since $x_{ij}n_j \leq \gamma_r n_i$,
    the possible values of $x_{ij}^{(r)}$ are $0,1,2,\ldots, \big\lfloor \frac{\gamma_r n_i}{n_j} \big\rfloor$.
    Let $w = t\big\lfloor \frac{\gamma_r n_i}{n_j} \big\rfloor$ and rescale so that $t \in [0,1]$.
    After rescaling, the implicitly defined singular probability measures of the previous proposition converge weakly
    to the absolutely continuous probability measure on $[0,1]$ with density $(k-1)( 1-t)^{k-2}$; that is, the beta distribution with parameters $\alpha = 1$ and $\beta = k-1$.
\end{remark}

\begin{example}\label{ex:edge-weights}
    Let $S=\langle 6,14,20,23,24,30,31\rangle$ and $\vec{g}=(1,2,8,64,\dots)$. 
    Proposition \ref{Proposition:Moments} says that $\mu_1(x_{3,1}^{(4)})$ is approximately $ 160/21\approx 7.619047$, close to the actual value $\mu_1(x_{3,1}^{(4)})=7.60967$.
    Similarly, the probability $\P\big(x_{3,1}^{(4)} = 0 \big)$ that there is no edge between $g_4\cdot 6$ and $g_5\cdot 20$, is approximately $9/80=11.25$; see Figure \ref{fig:edge-weights}.
\end{example}

\begin{figure}
    \centering
    \includegraphics[width=0.75\linewidth]{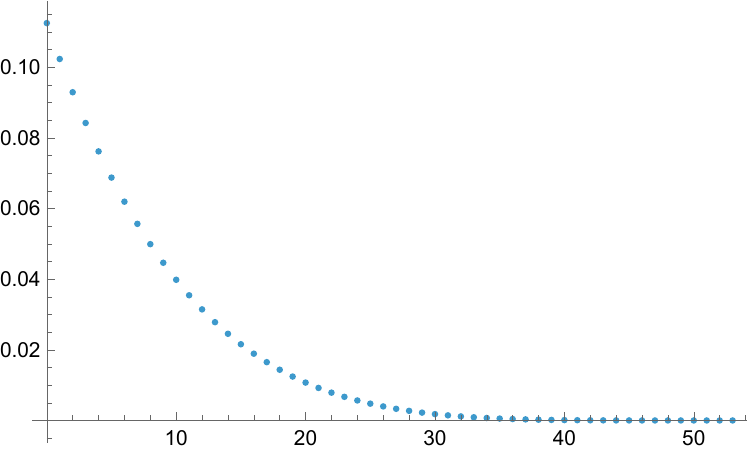}
    \caption{Probability distribution of $x_{3,1}^{(4)}$ from Example \ref{ex:edge-weights}.}
    \label{fig:edge-weights}
\end{figure}

The next result tells us that, under certain lenient circumstances, most algebras in a given ensemble are simple.

\begin{theorem}\label{Theorem:Simple}
    Suppose that $\sum_{r=1}^{\infty} \gamma_r^{-1} < \infty$. 
    Then a randomly chosen algebra in $\eng$ is simple with probability $1$.
\end{theorem}

\begin{proof}
    Consider the $r$th level of the Bratteli diagram associated to a random element of $\eng$.    
    Proposition \ref{Proposition:EdgeWeights} with $w=0$ implies that the asymptotic probability that there is no edge from $n_j$ to $n_i$ at the $r$th level is
    \begin{equation}\label{eq:nonzero-inj}
    \frac{n_j(k-1)}{\gamma_r n_i}.
    \end{equation}
    Let $E_r$ denote the event that there exist $1 \leq i,j \leq k$ such that $(r,n_j)$ is not adjacent to $(r+1,n_i)$, that is, the event that the Bratelli diagram at the $r$th level (ignoring the weights) is not a complete bipartite graph on two sets of $k$ vertices. As $r\to\infty$, \eqref{eq:nonzero-inj} ensures that
    \begin{equation*}
        \P(E_r) \lesssim \sum_{i,j=1}^k \frac{n_j(k-1)}{\gamma_r n_i} = \frac{c}{\gamma_r}.
    \end{equation*}    
    Since $\sum_{r=1}^{\infty} \P(E_r)$ converges, the Borel--Cantelli lemma ensures that the probability that $E_r$ occurs infinitely many times is zero.  Telescoping finitely many initial levels in a Bratteli diagram does not change the isomorphism class of an AF algebra, so with probability $1$, a randomly selected algebra in $\eng$ is simple.
\end{proof}

\begin{remark}
The product of \eqref{eq:nonzero-inj} over all generators
yields the probability 
\begin{equation*}
\frac{(k-1)^kn_j^k}{(n_1 n_2\cdots n_k)\gamma_r^k}
\end{equation*}
that $n_j$ vanishes at the $r$th level.  This
agrees with Proposition \ref{Proposition:AllVanish} when $M=\{n_j\}$.    
\end{remark}

\begin{remark}
    If $\sum_{r=1}^{\infty} \P(E_r)$ diverges, then the (second) Borel--Cantelli lemma says that with probability $1$,
    infinitely many of the $E_r$ occur.  However, this does not imply that the resulting AF algebras are not simple with
    probability $1$.  For example, the Bratelli diagram
    \begin{equation*}
        \begin{tikzcd}[cramped]
        	\bullet & \bullet & \bullet & \bullet & \bullet  \cdots \\
        	\bullet & \bullet & \bullet & \bullet & \bullet  \cdots
        	\arrow[from=1-1, to=2-2]
        	\arrow[from=1-2, to=1-3]
        	\arrow[from=1-2, to=2-3]
        	\arrow[from=1-3, to=2-4]
        	\arrow[from=1-4, to=1-5]
        	\arrow[from=1-4, to=2-5]
        	\arrow[from=2-1, to=1-2]
        	\arrow[from=2-1, to=2-2]
        	\arrow[from=2-2, to=1-3]
        	\arrow[from=2-2, to=2-3]
        	\arrow[from=2-3, to=1-4]
        	\arrow[from=2-3, to=2-4]
        	\arrow[from=2-4, to=1-5]
        	\arrow[from=2-4, to=2-5]
        \end{tikzcd}        
    \end{equation*}
    yields a simple AF algebra, even though $\P(E_r)=1$ for infinitely many $r$.
\end{remark}

Theorem \ref{Theorem:Simple} and the previous remark suggest an open problem
(Problem \ref{Problem:Simple})

When $\gamma_r^{-1}$ is summable, each ensemble $\eng$ consists almost entirely of simple AF algebras.
We would like to know about the properties of these algebras.  For example, are there infinitely many non-isomorphic algebras in each $\eng$?  We tackle this question in Section \ref{Section:K-Theory}.  Before that, however, we can identify a finite number of non-isomorphic algebras in $\eng$ by studying their ideals.  As Theorem \ref{Theorem:Simple} indicates, having any nontrivial ideals at all is rare.
    
\section{Ideals and UHF algebras}\label{Section:Ideals}

We now describe the ideals of chains in $\eng$. Although almost all rapidly growing AF algebras are simple when $\gamma_r^{-1}$ is summable (Theorem \ref{Theorem:Simple}), even in this context one can construct algebras with different numbers of ideals. This is an important step, since it establishes that $\eng$ contains many non-isomorphic algebras.  It would be silly indeed if every algebra in $\eng$ were isomorphic!  In what follows, $\cong$ denotes the (unital) *-isomorphisms of C*-algebras.

\begin{proposition}\label{prop:UHF-algs-and-simple}
    Suppose that $\gcd(n_1,n_2,\ldots,n_k)=1$. For all $M\subseteq \{1,2,\ldots, k\}$, there exists an $\A\in \eng$
    such that
    \begin{equation*}
    \A\cong \mathcal{V}\oplus \bigoplus_{j\in M}\mathcal{N}_j,
    \end{equation*}
    in which $\mathcal{V}$ is a simple AF-algebra and, for each $j\in \{1,2,\ldots,k\}$, $\mathcal{N}_j$ denotes the UHF algebra with diagram
    \begin{equation*}
    n_jg_1 \overset{\gamma_1}{\longrightarrow} n_jg_2
    \overset{\gamma_2}{\longrightarrow} n_jg_3 \overset{\gamma_3}{\longrightarrow} \cdots
    \end{equation*}
\end{proposition}

\begin{proof}
    Given $M$, construct the Bratteli diagram of $\A$ as follows. At each level, if $j\in M$, put an edge with weight $\gamma_r$ from $(r,j)$ to $(r+1,j)$.
    Let $p=|M^c|$. For all $i\in M^c$, there exists some $N\in\Z_{\geq 1}$ such that
    \begin{equation*}
    \gamma_r\cdot \min\{n_i:  i\in M^c\}-\sum_{i\in M^c}n_i
    \end{equation*}
    is greater than the Frobenius number of $\langle n_i \rangle_{i\in M^c}$ for all $r\geq N$. For $r< N$, at each level put an edge with weight $\gamma_r$ from $(r,i)$ to $(r+1,i)$. If $r\geq N$, for any $j\in M^c$, there exist $s_{j,1},s_{j,2},\ldots,s_{j,p}\in \Z_{\geq 0}$ such that
    \begin{equation*}
    \sum_{i\in M^c}s_{j,i}n_i=\gamma_rn_j-\sum_{i\in M^c}n_i
    \end{equation*}
    or, equivalently, $\sum_{i\in M^c}(s_{j,i}+1)n_i=\gamma_rn_j.$

    Put an edge with weight $s_{j,i}+1$ from $(r,i)$ to $(r+1,j)$ for all $i,j\in M^c$. Then the graph with vertices $\{(r,i),(r+1,i) :  i\in M^c\}$ is a complete bipartite graph on two sets of $p$ vertices for all $r$. Since the diagram is connected for $r\geq N$, it yields a simple subalgebra $\mathcal{V}$.

    We see that the subalgebra of $\A$ associated to the subdiagram with vertices $\{(r,i) :  i\in M^c,r\in \N\}$ is an ideal. Furthermore, it is clear that for fixed $j\in M$ the subalgebra associated to the subdiagram with vertices $\{(r,j) :  r\in \N\}$ is an ideal isomorphic to the indicated UHF algebra $\mathcal{N}_j$.
\end{proof}

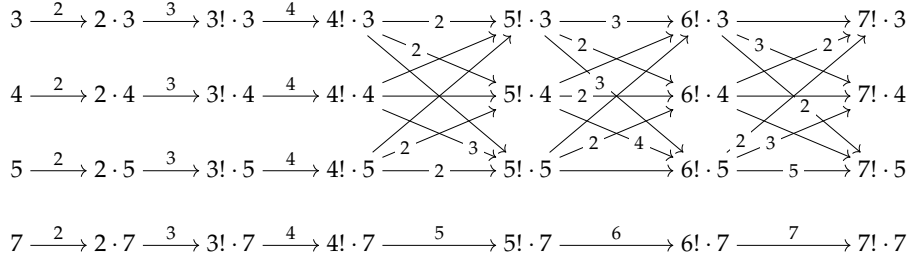
\begin{figure}\small
    \centering
\[\begin{tikzcd}[cramped]
	3 & {2\cdot 3} & {3!\cdot 3} & {4!\cdot 3} && {5!\cdot 3} && {6!\cdot 3} && {7!\cdot 3} \\
	4 & {2\cdot 4} & {3!\cdot 4} & {4!\cdot 4} && {5!\cdot 4} && {6!\cdot 4} && {7!\cdot 4} \\
	5 & {2\cdot 5} & {3!\cdot 5} & {4!\cdot 5} && {5!\cdot 5} && {6!\cdot 5} && {7!\cdot 5} \\
	7 & {2\cdot 7} & {3!\cdot 7} & {4!\cdot 7} && {5!\cdot 7} && {6!\cdot 7} && {7!\cdot 7}
	\arrow["2", from=1-1, to=1-2]
	\arrow["3", from=1-2, to=1-3]
	\arrow["4", from=1-3, to=1-4]
	\arrow["2"{description}, from=1-4, to=1-6]
	\arrow["2"{description, pos=0.3}, from=1-4, to=2-6]
	\arrow[from=1-4, to=3-6]
	\arrow["3"{description}, from=1-6, to=1-8]
	\arrow["2"{description, pos=0.2}, from=1-6, to=2-8]
	\arrow["3"{description, pos=0.4}, from=1-6, to=3-8]
	\arrow[from=1-8, to=1-10]
	\arrow["3"{description, pos=0.2}, from=1-8, to=2-10]
	\arrow["2"{description, pos=0.6}, from=1-8, to=3-10]
	\arrow["2", from=2-1, to=2-2]
	\arrow["3", from=2-2, to=2-3]
	\arrow["4", from=2-3, to=2-4]
	\arrow[from=2-4, to=1-6]
	\arrow[from=2-4, to=2-6]
	\arrow["3"{description, pos=0.8}, from=2-4, to=3-6]
	\arrow[from=2-6, to=1-8]
	\arrow["2"{description, pos=0.2}, from=2-6, to=2-8]
	\arrow["4"{description, pos=0.7}, from=2-6, to=3-8]
	\arrow["2"{description, pos=0.8}, from=2-8, to=1-10]
	\arrow[from=2-8, to=2-10]
	\arrow[from=2-8, to=3-10]
	\arrow["2", from=3-1, to=3-2]
	\arrow["3", from=3-2, to=3-3]
	\arrow["4", from=3-3, to=3-4]
	\arrow[from=3-4, to=1-6]
	\arrow["2"{description, pos=0.2}, from=3-4, to=2-6]
	\arrow["2"{description}, from=3-4, to=3-6]
	\arrow[from=3-6, to=1-8]
	\arrow["2"{description, pos=0.3}, from=3-6, to=2-8]
	\arrow[from=3-6, to=3-8]
	\arrow["2"{description, pos=0.1}, from=3-8, to=1-10]
	\arrow["3"{description, pos=0.3}, from=3-8, to=2-10]
	\arrow["5"{description}, from=3-8, to=3-10]
	\arrow["2", from=4-1, to=4-2]
	\arrow["3", from=4-2, to=4-3]
	\arrow["4", from=4-3, to=4-4]
	\arrow["5", from=4-4, to=4-6]
	\arrow["6", from=4-6, to=4-8]
	\arrow["7", from=4-8, to=4-10]
\end{tikzcd}\]
    \caption{An illustration of the construction in Proposition \ref{prop:UHF-algs-and-simple} for $\langle 3,4,5,7\rangle$, in which $\{n_m: m\in M\}=\{7\}$.}
    \label{fig:UHF-algs-and-simple-proof}
\end{figure}

Figure \ref{fig:UHF-algs-and-simple-proof} concerns an example of this construction.

\begin{corollary}\label{cor:k-nonisom-algs}
    The set $\eng$ contains at least $k$ non-isomorphic algebras. 
\end{corollary}

\begin{proof}
Let $d=\gcd(n_1,n_2,\ldots,n_k)$, and if $\mathbf{g}=(g_1,g_2,\ldots)$ let $\mathbf{g}'=(d,g_1,g_2,\ldots)$. For all $m\in \{1,2,\ldots,k\}$, Proposition \ref{prop:UHF-algs-and-simple} provides an algebra in $\E(\n/d,\mathbf{g}')$ with exactly $m$ ideals. We may assume the first partial-multiplicity matrix of this algebra is $dI_k$. Removing the first level, we obtain an algebra in $\eng$ with exactly $m$ ideals. Thus, there are at least $k$ non-isomorphic algebras in $\eng$.
\end{proof}

More generally, the elements of $\eng$ often have interesting subalgebras. For example, the growth sequence poses constraints on embeddings of UHF algebras in algebras in $\eng$. This suggests an  open problem (Problem \ref{Problem:UHF})

\begin{figure}
\[\begin{tikzcd}
	2 & 4 & 32 & \dots \\
	7 & 14 & 112 & \dots
	\arrow["2", color=red, from=1-1, to=1-2]
	\arrow["{2^2}", color=red, from=1-2, to=1-3]
	\arrow["{2^4}", color=red, from=1-3, to=1-4]
	\arrow["2", from=2-1, to=2-2]
	\arrow["{2^2}", from=2-2, to=2-3]
	\arrow["{2^4}", from=2-3, to=2-4]
\end{tikzcd}\]
    \caption{The algebra $\A\in \eng$ shown above is the direct sum of the CAR algebra and the UHF algebra with generalized integer $7\cdot 2^{\infty}$. We can recover the CAR algebra from $\A$ by following the red path.}
    \label{fig:car-alg}
\end{figure}

\begin{example}
If $g_r=2^{2^{(r-1)}-1}$ and $\n=\langle2,7\rangle$, we can recover the CAR algebra \cite[Ex.~III.2.4]{davidson}; see Figure~\ref{fig:car-alg}.    
This leads to another interesting question: for which $\eng$ can the CAR algebra be embedded?
\end{example}

In general, ideals are not a fine enough tool to distinguish non-isomorphic algebras in $\eng$. In the next section, we use K-theory to show that, in many circumstances, $\eng$ contains uncountably many non-isomorphic algebras.

\section{Isomorphism Classes of Families of AF Algebras}\label{Section:K-Theory}

In this section, we determine the K-theory in the sense of Elliott \cite{Elliott76} for some classes of AF algebras  arising from numerical semigroups. We use this to show that for certain $(\n,\bg)$, the ensemble $\eng$ contains uncountably many isomorphism classes of AF algebras. In fact, our argument shows that there exist uncountably many non-Morita equivalent AF algebras for these particular kinds of ensembles. We suspect that this holds more generally, but relaxing the conditions on $(\n,\bg)$ makes calculating the K-theory more difficult.


The simplest ensembles of rapidly growing AF algebras arise when $k=2$ and $n_1=n_2 = n$.  
In this setting, we have the following result that suggests the richness of the ensembles $\E(\twovector{n}{n},\bg)$.

\begin{theorem}\label{thm:inf-identical-gens}
    For each $n\in \N$ and growth sequence $\bg$, the ensemble $\E(\twovector{n}{n},\bg)$ contains uncountably many isomorphism classes of AF algebras.
\end{theorem}

To prove this theorem, we classify the K-theory of some chains in $\E(\twovector{n}{n},\bg)$. To classify the $K_0$ group, we need a definition \cite[Sect.~3.1 and 4.4]{dixmier}.

\begin{definition}\label{def:generalized-integer}
    Let $\P$ denote the set of prime numbers. A \emph{generalized integer} is a function $i:\P\rightarrow \N\cup\{+\infty\}$, usually written as a formal product $\prod_{p\in \P}p^{i(p)}$ \cite[p. 193]{dixmier}. Generalized integers $m=\prod_{p\in \P}p^{i(p)}$ and $n=\prod_{p\in \P}p^{j(p)}$ are \emph{equivalent}, denoted $m\sim n$, if $i(p)=j(p)$ with finitely many exceptions, and if $i(p),j(p)<\infty$ whenever $i(p)\neq j(p)$.
    For each $n=\prod_{p\in \P}p^{j(p)}$, we write $\Q_n$ for the subgroup of $\Q$ consisting of rational numbers whose denominators formally  divide $n$.
\end{definition}

Observe that natural numbers are special cases of generalized integers.

\begin{lemma}\label{lem:isomorphism-of-direct-sum}
    Let $m=\prod_{p\in\P}p^{i(p)}$, $n=\prod_{p\in\P}p^{j(p)}$, $m'=\prod_{p\in\P}p^{i'(p)}$, and $n'=\prod_{p\in\P}p^{j'(p)}$ be generalized integers. 
    Then $\Q_m\oplus \Q_n\cong \Q_{m'}\oplus \Q_{n'}$ if and only if $m\sim m'$ and $n\sim n'$ (or $m\sim n'$ and $n\sim m'$).
\end{lemma}

\begin{proof}
    First suppose that $\Q_m\oplus\Q_n\rightarrow\Q_{m'}\oplus \Q_{n'}$ is an isomorphism. Let $(1,0)\mapsto(a,c)$ and $(0,1)\mapsto(b,d)$. We consider the case  $a,b,c,d\neq 0$. Let $p_0\in P$ be such that all primes $p\geq p_0$ are relatively prime to $a,b,c,d,$ and $ad-bc$. Fix $p\geq p_0$. Since $(1,0)$ is divisible by $p^{i(p)}$, so are $a$ and $c$. Since $a\in \Q_{m'}$ and $c\in \Q_{n'}$, and $p$ is relatively prime to $a$ and $c$, we have $i(p)\leq i'(p),j'(p)$. Similarly, $j(p)\leq i'(p),j'(p)$. The inverse map $\Q_{m'}\oplus \Q_{n'}\rightarrow \Q_m\oplus \Q_n$ is defined by \[(1,0)\mapsto \frac{1}{ad-bc}(d,-c)\] and \[(0,1)\mapsto \frac{1}{ad-bc}(-b,a).\] Therefore, the same reasoning shows that $i'(p),j'(p)\leq i(p),j(p)$. It follows that $i(p)=j(p)=i'(p)=j'(p)$. Similar reasoning shows that for any $p\in\P$, if any of $i(p),j(p),i'(p),j'(p)$ equal $+\infty$, then all equal $+\infty$. Therefore, $m,m',n,n'$ are all equivalent. Next, if for example $c=0$, then $a,d\neq 0$. Similar reasoning to the above shows that $m\sim m'$ and $n\sim n'$. If $a=0$, we find that $m\sim n'$ and $n\sim m'$. 
    If $b=0$, then $b,c\neq 0$, so $m\sim m'$ and $n\sim n'$. Finally, if $d=0$, then $m\sim n'$ and $n\sim m'$.

    Finally, the converse is straightforward.
\end{proof}

 To construct the AF algebras in $\E(\twovector{n}{n},{\bf g})$, we begin with an integer sequence $\g_r>2$ such that~\eqref{eq:Rapid} holds. Let $\frac{1}{2}\g_r<p_r<\gamma_r$; this requirement is motivated by \cite{marechal}. It follows that 
\begin{align*}
    0&<\g_r-p_r<p_r,\\
    0&<2p_r-\g_r<p_r.
\end{align*}
Define an AF algebra via multiplicity matrices 
\begin{equation}\label{eq:X_r-of-identical-gens}
    X_r=\begin{bmatrix}
        p_r&\g_r-p_r\\
        \g_r-p_r&p_r
    \end{bmatrix}
\end{equation} for $r\geq 1$. Let $B_r=X_1 X_2\cdots X_r$. Then the $K_0$-group is $\bigcup_{r\geq 1}B_r^{-1}(\Z^2)$. 

Let $q_r=2p_r-\g_r$, and set $g_r=\prod_{i=1}^r \g_r$ and $a_r=\prod_{i=1}^r q_r$. Note that $\g_r$ and $q_r$ have the same parity, hence the same is true for $g_r$ and $a_r$. Since $p_r=\frac{1}{2}(\g_r+q_r)$, the algebra may equally well be defined by the sequence $(q_r)$ with $0 < q_r <\g_r$. We summarize a result from \cite{marechal} to describe the $K_0$-group. Let
\begin{equation*}
    I=\begin{bmatrix}
        1&0\\0&1
    \end{bmatrix},
    \quad
    V=\begin{bmatrix}
        0&1\\
        1&0
    \end{bmatrix},
    \quad
    P=\frac{1}{2}(I+V),
    \quad \text{and} \quad
    P^\perp=\frac{1}{2}(I-V).
\end{equation*}
Then $P$ and $P^\perp=I-P$ are orthogonal projections and 
\begin{align*}
    X_r&=p_rI+(\g_r-p_r)V\\
    &=p_r(P+P^\perp)+(\g_r-p_r)(P-P^\perp)\\
    &=\g_rP+(2p_r-\g_r)P^\perp\\
    &=\g_rP+q_rP^\perp,
\end{align*}
so
\begin{equation*}
    B_r=\bigg(\prod_{i=1}^r\g_i\bigg)P+\bigg(\prod_{i=1}^rq_i\bigg)P^\perp=g_rP+a_rP^\perp    
\end{equation*}
and
\begin{equation*}
    B_r^{-1}=g_r^{-1}P+a_r^{-1}P^\perp.    
\end{equation*}
Let $\vec{e}_1,\vec{e}_2$ denote the standard basis vectors of $\C^2$. Define
\begin{equation*}
    \vec{u}_1=\begin{bsmallmatrix}
        1\\1
    \end{bsmallmatrix}=2P\be_1 \qquad\text{and}\qquad
    \vec{u}_2=\begin{bsmallmatrix}
        1\\-1
    \end{bsmallmatrix}=2P^\perp\be_2,
\end{equation*}
so that
\begin{equation*}
        \be_1=\frac{1}{2}(\vec{u}_1+\vec{u}_2) \qquad\text{and}\qquad
    \be_2=\frac{1}{2}(\vec{u}_1-\vec{u}_2)
\end{equation*}
and
\begin{equation*}
    B_r^{-1}\be_1 =(2a_rg_r)^{-1}(a_r{\vec{u}_1}+g_r\vec{u}_2) \quad\text{and}\quad
    B_r^{-1}\be_2=(2a_rg_r)^{-1}(a_r{\vec{u}_1}-g_r\vec{u}_2).
\end{equation*}
Then
\begin{align*}
    B_r^{-1}(\Z^2)&=\operatorname{span}_{\Z}\{B_r^{-1}\be_1,B_r^{-1}\be_2\}\\
    &=\{(2a_rg_r)^{-1}(ka_r{\vec{u}_1}+kg_r{\vec{u}_2}+\ell a_r{\vec{u}_1}-\ell g_r{\vec{u}_2}): k,\ell\in \Z\}\\
    &=\frac{1}{2}\{(k+\ell)g_r^{-1}{\vec{u}_1}+(k-\ell)a_r^{-1}{\vec{u}_2}: k,\ell \in \Z\}.
\end{align*}
One can check that
\begin{align*}
    \{k+\ell,k-\ell: k,\ell\in \Z\}&=\{(b,c): \text{b,c} \in \Z \text{ and } b+c \text{ is even}\}\\
    &=2\Z^2+\Z\vec{u}_1.
\end{align*}
Therefore, 
\begin{align*}
    B_r^{-1}(\Z^2)&=\frac{1}{2}\{bg_r^{-1}{\vec{u}_1}+ca_r^{-1}{\vec{u}_2}: (b,c)\in 2\Z^2+\Z {\vec{u}_1}\}\\
    &=\{b'g_r^{-1}\vec{u}_1+c'a_r^{-1}{\vec{u}_2}: (b',c')\in \Z^2+\Z\cdot\frac{1}{2}{\vec{u}_1}\}.
\end{align*}
Let
\begin{align*}
H&=\bigcup_r\{(bg_r^{-1},ca_r^{-1}:(b,c)\in \Z^2+\Z\cdot\frac{1}{2}\vec{u}_1\},
\end{align*}
which is isomorphic to the $K_0$-group of the AF algebra. 

Consider the generalized integers $g_\infty=\prod_r\g_r$ and $a_\infty=\prod_rq_r$. Since $\g_r$ and $q_r$ have the same parity, the same is true of $g_\infty$ and $a_\infty$ (and that both $g_\infty$ and $a_\infty$ are divisible by $2^\infty$ or neither is).

\begin{theorem}\label{thm:2-symm-with-arbitrary-ints}
    Let $m,n$ be generalized integers. Suppose that $m$ is not an integer, that $m$ and $n$ have the same parity, and that if $m$ is divisible by $2^\infty$ then so is $n$. Then there are sequences $\g_r$ and $p_r$ as above such that the AF algebra constructed as above satisfies $g_\infty=m$ and $a_\infty=n$.
\end{theorem}

\begin{proof}

    Write $m=\prod_{r=1}^\infty m_r$ and $n=\prod_{r=1}^\infty n_r$ so that for all $r$, we have $1\leq n_r<m_r$, $\sum_r m_r^{-2}<\infty$, and $m_r,n_r$ have the same parity. Set $g_r=m_r$ and $p_r=\frac{1}{2}(m_r+n_r)$. Then $\frac{1}{2}g_r<p_r<g_r$, and $q_r=2p_r-\g_r=n_r$. Then $g_\infty=\prod_r g_r=m$ and $a_\infty=\prod_i q_i=n$.
\end{proof}

Next, note that $\Q_{g_\infty}\oplus \Q_{a_\infty}\subseteq H\subseteq \Q_{2g_\infty}\oplus \Q_{2a_\infty}$.

\begin{proposition}\label{prop:cong-dim-group-implies-ints-equiv}
    Let $g_r,p_r$ and $g_r',p_r'$ define two AF algebras as constructed above, with $K_0$-groups $H$ and $H'$, respectively. If $H\cong H'$, then $g_\infty\sim g_\infty'$ and $a_\infty\sim a_\infty'$ (or $g_\infty\sim a_\infty'$ and $a_\infty\sim g_\infty'$).
\end{proposition}

\begin{proof}
    Suppose that $H\cong H'$. We know that $\Q_{g_\infty}\oplus \Q_{a_\infty}\subseteq H\subseteq \Q_{2g_\infty}\oplus \Q_{2a_\infty},$ and $\Q_{g'_\infty}\oplus \Q_{a'_\infty}\subseteq H'\subseteq \Q_{2g'_\infty}\oplus \Q_{2a'_\infty}$. Let $H\rightarrow \Q_{2g'_\infty}\oplus \Q_{2a'_\infty}$ be the composition of the inclusion with the isomorphism. Let $(1,0)\mapsto (a,c)$ and $(0,1)\mapsto (b,d)$. If $p\in\P$ is relatively prime to $2,a,b,c,d,ad-bc$, then the argument in the proof of Lemma~\ref{lem:isomorphism-of-direct-sum} shows that $i(p)=i'(p)$ and $j(p)=j'(p)$ or $i(p)=j'(p)$ and $j(p)=i'(p)$.
\end{proof}

We are now ready to prove Theorem~\ref{thm:inf-identical-gens}.
\begin{proof}[Proof of Theorem~\ref{thm:inf-identical-gens}.]
Let $\E(\twovector{n}{n},\bg)$ be given. Since $\sum_r \g_r^{-2}<\infty$, it must be that $\lim_{r\rightarrow\infty}\g_r=\infty$. Therefore, there are uncountably many nonequivalent generalized integers $m=\prod_{r=1}^\infty m_r$ such that $m_r$ is a prime satisfying $0<m_r<\g_r$. Therefore, let $q_r=m_r$ so that $a_\infty=m$; the condition that $0<q_r<\g_r$ is satisfied.
By Proposition~\ref{prop:cong-dim-group-implies-ints-equiv}, if $a_\infty'\not\sim a_\infty$, then $H\not\cong H'$. Therefore,  as there are uncountably many nonequivalent generalized integers $a_\infty$, there are uncountably many nonisomorphic dimension groups for such AF algebras, so there are uncountably many nonisomorphic AF algebras in the ensemble $\E(\twovector{n}{n},\bg)$.
\end{proof}

\begin{remark}
    In fact, AF algebras with nonisomorphic $K_0$-groups are not Morita equivalent (see \cite[Corollary 6.2.11]{Wegge-Olsen}). Therefore Theorem~\ref{thm:inf-identical-gens} can be improved by replacing ‘non-isomorphism’ by 'non-Morita equivalence'.
\end{remark}

A possible generalization demands further exploration (Problem \ref{Problem:Repeated}):
if at least one generator is repeated, does $\eng$ contain uncountably many nonisomorphic AF algebras?

Next we consider whether the statistical properties of an $\A \in \E(\twovector{n}{n},\bg)$ can be used to distinguish it from non-isomorphic algebras in the same ensemble.  We first require the next result.  
Since the entries of \eqref{eq:X_r-of-identical-gens} depend only upon the $(1,1)$ entry (we regard $\gamma_r$ as given), we formulate things in terms of this entry.

\begin{proposition}\label{prop:nonisom-using-p}
    Let $\g_r,p_r$ and $\g_r',p_r'$ define two AF algebras as constructed above with $K_0$ groups $H,H'$, respectively. Let
    $g_\infty=\prod_r\g_r$ and $\alpha_\infty=\prod_r q_r=\prod_r (2p_r-\g_r)$,  and let $g_\infty'=\prod_r\g_r'$ and $\alpha_\infty'=\prod_rq_r'=\prod_r(2p_r'-\g_r')$.
    Let $p$ be an odd prime such that $p^{\infty}$ formally divides $g_{\infty}$. Let $X_r$ and $Y_r$ be the partial multiplicity matrices for $\A,\A'\in \E(\twovector{n}{n},\bg)$, respectively, at the $r$th level.
    If $p\mid x_{11}^{(r)}$ and $p\nmid y_{11}^{(r)}$ for all but finitely many $r$, then $\A\not\cong \A'$.
\end{proposition}

\begin{proof}
    Since $p$ is an odd prime dividing $\g_r$ all but finitely often, for such indices it follows that $p\mid(2z-\g_r)$ if and only if $p\mid z$ for $z=x_{11}^{(r)}$ or $y_{11}^{(r)}$. Thus, if $p\mid x_{11}^{(r)}$ and $p\mid y_{11}^{(r)}$ all but finitely often, then $p^\infty$ formally divides $\alpha_\infty$ and does not formally divide $\alpha_\infty'$. Thus, $G\not\cong H$ by Proposition~\ref{prop:cong-dim-group-implies-ints-equiv}.
\end{proof}

We next compute the asymptotic probability that $p$ divides $x_{11}^{(r)}$.

\begin{proposition}\label{prop:prob-p-mid-x}
    If $p$ is an odd prime such that $p^{\infty}$ formally divides $g_{\infty}$, then as $r\to\infty$
    \begin{equation*}
        \P\Big(p\mid x_{11}^{(r)}\Big) \sim \frac{1}{p}.
    \end{equation*}
    In particular, the probability that $p\mid x_{11}^{(r)}$ for all but finitely many $r$ is $0$.
\end{proposition}

\begin{proof}
For any $w \in \Z_{\geq 0}$, Proposition~\ref{Proposition:EdgeWeights} ensures that
\begin{equation}\label{eq:prob-equally-likely}
\P\Big(x_{11}^{(r)}=w\Big)\sim\frac{1}{\g_r}.
\end{equation}
Thus, Lemma \ref{Lemma:AsympototicSum} yields
\begin{equation*}
    \P\Big(p\mid x_{11}^{(r)}\Big)
    \sim \bigg\lfloor \frac{\g_r}{p} \bigg\rfloor \frac{1}{\g_r} \sim \frac{1}{p}.
\end{equation*}
For the second statement, apply the second Borel--Cantelli lemma \cite[Thm.~4.4]{Billingsley}
to the complementary events $p \nmid x_{11}^{(r)}$.
\end{proof}

\begin{theorem}\label{Theorem:PrimeThing}
    Suppose that $p$ is an odd prime such that $p^{\infty}$ formally divides $g_{\infty}$.
    Let $\A \in \E(\twovector{n}{n},\bg)$ be such that $p\mid x_{11}^{(r)}$ all but finitely often. 
    Then the probability that a randomly selected $\A'\in \E(\twovector{n}{n},\bg)$ is isomorphic to $\A$ is zero.
\end{theorem}

\begin{proof}
This follows immediately from Propositions \ref{prop:nonisom-using-p} and \ref{prop:prob-p-mid-x}.
\end{proof}

\section{Further Research}\label{Section:Further}
We conclude with several appealing questions that we hope will spur further research on rapidly growing AF algebras.

As we have seen, rapidly growing AF algebras are tractable generalizations of UHF algebras that link probability theory, combinatorics, and operator algebras.  
We know that certain ensembles $\eng$ contain uncountably many isomorphism classes of algebras.
One hopes to show that there are infinitely (perhaps uncountably) many non-isomorphic algebras in any ensemble $\eng$ without the restrictions in Theorem~\ref{thm:inf-identical-gens}. 
Although it is not clear that these restrictions are truly necessary, we were unable to obtain a general result that applies to all ensembles.  Therefore, we pose the following question.

\begin{problem}
When does $\eng$ contain infinitely (or uncountably) many non-isomorphic AF algebras?  Does every ensemble $\eng$ have this property?
\end{problem}

Theorem \ref{thm:inf-identical-gens} suggests that the next problem should have an affirmative solution.
However, it is not immediately obvious that adding generators does not reduce the number of isomorphism classes
represented in the original ensemble.

\begin{problem}\label{Problem:Repeated}
If at least one generator is repeated, does $\eng$ contain uncountably many non-isomorphic AF algebras?    
\end{problem}

As we have seen, K-theory distinguishes between some members of certain ensembles $\eng$.  One wonders how much further one can push these techniques.

\begin{problem}
Explicitly describe the K-theory of rapidly growing AF algebras.
\end{problem}

Perhaps K-theory it is a finer tool than necessary.  We might seek not to distinguish a given rapidly growing AF algebra from all other AF algebras, but rather from non-isomorphic algebras in a given ensemble.  Therefore, other techniques might serve the same purpose.  Theorem \ref{Theorem:PrimeThing} may be a step in this direction.

\begin{problem}
    Is there a collection of asymptotic statistics that can distinguish between non-isomorphic AF algebras in a given ensemble? 
\end{problem}

Theorem \ref{Theorem:Simple} immediately suggests the next question.

\begin{problem}\label{Problem:Simple}
    If $\sum_{r=1}^{\infty} \gamma_r^{-1}$ diverges, what is the probability that a randomly chosen algebra in $\eng$ is simple?
\end{problem}

Similarly, the results of Section \ref{Section:Ideals} points toward the following.

\begin{problem}\label{Problem:UHF}
    Classify all embeddings of UHF algebras into algebras in $\eng$.
\end{problem}

We hope to investigate these sorts of problems in later work.

\bibliographystyle{plain}
\bibliography{supporting_files/bib}

@incollection{bouthat2024hunterspositivitytheoremrandom,
    AUTHOR = {Bouthat, Ludovick and Ch\'{a}vez, \'{A}ngel and Garcia,
              Stephan Ramon},
     TITLE = {Hunter's positivity theorem and random vector norms},
 BOOKTITLE = {Operator theory, related fields, and applications},
    SERIES = {Oper. Theory Adv. Appl.},
    VOLUME = {307},
     PAGES = {149--215},
 PUBLISHER = {Birkh\"{a}user/Springer, Cham},
      YEAR = {2025},
      ISBN = {978-3-032-00154-2; 978-3-032-00155-9},
   MRCLASS = {15 (05 60)},
  MRNUMBER = {4978054},
       DOI = {10.1007/978-3-032-00155-9\{_}6}

@book {HalmosMeasure,
    AUTHOR = {Halmos, Paul R.},
     TITLE = {Measure {T}heory},
 PUBLISHER = {D. Van Nostrand Co., Inc., New York},
      YEAR = {1950},
     PAGES = {xi+304},
   MRCLASS = {27.2X},
  MRNUMBER = {33869},
MRREVIEWER = {S. Kakutani},
}

@incollection {ReciprocalSchur,
    AUTHOR = {B\"{o}ttcher, Albrecht and Garcia, Stephan Ramon and Mitkovski,
              Mishko},
     TITLE = {The reciprocal {S}chur inequality},
 BOOKTITLE = {Analysis without borders},
    SERIES = {Oper. Theory Adv. Appl.},
    VOLUME = {297},
     PAGES = {41--49},
 PUBLISHER = {Birkh\"{a}user/Springer, Cham},
      YEAR = {[2024] \copyright 2024},
   MRCLASS = {26D05 (05A20 05E05 26D10 39B62)},
  MRNUMBER = {4786513},
MRREVIEWER = {Jia Xu},
       DOI = {10.1007/978-3-031-59397-0\_3},
       URL = {https://doi.org/10.1007/978-3-031-59397-0_3},
}

@book {Davis,
    AUTHOR = {Davis, P. J.},
     TITLE = {Circulant matrices},
      NOTE = {A Wiley-Interscience Publication,
              Pure and Applied Mathematics},
 PUBLISHER = {John Wiley \& Sons, New York-Chichester-Brisbane},
      YEAR = {1979},
     PAGES = {xv+250},
      ISBN = {0-471-05771-1},
   MRCLASS = {15-02 (15A57 65F30)},
  MRNUMBER = {543191 (81a:15003)},
MRREVIEWER = {Cs. J. Heged{\H{u}}s},
}

@article {SchoenbergWhitney,
    AUTHOR = {Schoenberg, I. J. and Whitney, Anne},
     TITLE = {On {P}\'{o}lya frequence functions. {III}. {T}he positivity of
              translation determinants with an application to the
              interpolation problem by spline curves},
   JOURNAL = {Trans. Amer. Math. Soc.},
  FJOURNAL = {Transactions of the American Mathematical Society},
    VOLUME = {74},
      YEAR = {1953},
     PAGES = {246--259},
      ISSN = {0002-9947},
   MRCLASS = {27.0X},
  MRNUMBER = {53177},
MRREVIEWER = {E. Hille},
       DOI = {10.2307/1990881},
       URL = {https://doi.org/10.2307/1990881},
}

@book {CohnMeasure,
    AUTHOR = {Cohn, Donald L.},
     TITLE = {Measure theory},
    SERIES = {Birkh\"{a}user Advanced Texts: Basler Lehrb\"{u}cher. [Birkh\"{a}user
              Advanced Texts: Basel Textbooks]},
   EDITION = {Second},
 PUBLISHER = {Birkh\"{a}user/Springer, New York},
      YEAR = {2013},
     PAGES = {xxi+457},
      ISBN = {978-1-4614-6955-1; 978-1-4614-6956-8},
   MRCLASS = {28-01},
  MRNUMBER = {3098996},
MRREVIEWER = {Ville Suomala},
       DOI = {10.1007/978-1-4614-6956-8},
       URL = {https://doi.org/10.1007/978-1-4614-6956-8},
}

@article{abhyankar1967local,
  title={Local rings of high embedding dimension},
  author={Abhyankar, S. S.},
  journal={Amer. J. Math.},
  volume={89},
  number={4},
  pages={1073--1077},
  year={1967},
  publisher={JSTOR}
}

@book{barucci1997maximality,
  title={Maximality properties in numerical semigroups and applications to one-dimensional analytically irreducible local domains},
  author={Barucci, V. and Dobbs, D. E and Fontana, M.},
  volume={598},
  year={1997},
  publisher={American Mathematical Soc.}
}

@incollection{pisinger1998knapsack,
  title={Knapsack problems},
  author={Pisinger, David and Toth, Paolo},
  booktitle={Handbook of combinatorial optimization},
  pages={299--428},
  year={1998},
  publisher={Springer}
}

@book{de2013algebraic,
  title={Algebraic and geometric ideas in the theory of discrete optimization},
  author={De Loera, Jes{\'u}s A and Hemmecke, Raymond and Matthias, K and others},
  volume={14},
  year={2013},
  publisher={SIAM}
}

@book {Billingsley,
    AUTHOR = {Billingsley, Patrick},
     TITLE = {Probability and measure},
    SERIES = {Wiley Series in Probability and Mathematical Statistics},
   EDITION = {Third},
      NOTE = {A Wiley-Interscience Publication},
 PUBLISHER = {John Wiley \& Sons, Inc., New York},
      YEAR = {1995},
     PAGES = {xiv+593},
      ISBN = {0-471-00710-2},
   MRCLASS = {60-01 (28-01)},
  MRNUMBER = {1324786},
}

@article {garcia2025noncommutativegeneralizationhunterspositivity,
    AUTHOR = {Garcia, Stephan Ramon and Vol\v{c}i\v{c}, Jurij},
     TITLE = {A noncommutative generalization of {H}unter's positivity
              theorem},
   JOURNAL = {Proc. Amer. Math. Soc.},
  FJOURNAL = {Proceedings of the American Mathematical Society},
    VOLUME = {154},
      YEAR = {2026},
    NUMBER = {2},
     PAGES = {585--597},
      ISSN = {0002-9939,1088-6826},
   MRCLASS = {05E05 (13J30 15A45 16S10)},
  MRNUMBER = {5016543},
       DOI = {10.1090/proc/17480},
       URL = {https://doi.org/10.1090/proc/17480},
}

@article {Hunter,
    AUTHOR = {Hunter, David B.},
     TITLE = {The positive-definiteness of the complete symmetric functions
              of even order},
   JOURNAL = {Math. Proc. Cambridge Philos. Soc.},
  FJOURNAL = {Mathematical Proceedings of the Cambridge Philosophical
              Society},
    VOLUME = {82},
      YEAR = {1977},
    NUMBER = {2},
     PAGES = {255--258},
      ISSN = {0305-0041},
   MRCLASS = {05A15 (26A69 30A10)},
  MRNUMBER = {450079},
MRREVIEWER = {John P. Coleman},
       DOI = {10.1017/S030500410005386X},
       URL = {https://doi.org/10.1017/S030500410005386X},
}

@article {Tao,
    AUTHOR = {Tao, Terence},
    TITLE = {Schur convexity and positive definiteness of the even degree complete homogeneous symmetric polynomials},
NOTE = {\url{https://terrytao.wordpress.com/2017/08/06/schur-convexity-and-positive-definiteness-of-the-even-degree-complete-homogeneous-symmetric-polynomials/}}
}

@book {StanleyBook2,
    AUTHOR = {Stanley, Richard P.},
     TITLE = {Enumerative combinatorics. {V}ol. 2},
    SERIES = {Cambridge Studies in Advanced Mathematics},
    VOLUME = {62},
      NOTE = {With a foreword by Gian-Carlo Rota and appendix 1 by Sergey
              Fomin},
 PUBLISHER = {Cambridge University Press, Cambridge},
      YEAR = {1999},
     PAGES = {xii+581},
      ISBN = {0-521-56069-1; 0-521-78987-7},
   MRCLASS = {05A15 (05-02 05E05 05E10 68R05)},
  MRNUMBER = {1676282},
MRREVIEWER = {Ira Gessel},
       DOI = {10.1017/CBO9780511609589},
       URL = {https://doi.org/10.1017/CBO9780511609589},
}

@article {EffrosShen,
    AUTHOR = {Effros, Edward G. and Shen, Chao Liang},
     TITLE = {Approximately finite {$C\sp{\ast} $}-algebras and continued
              fractions},
   JOURNAL = {Indiana Univ. Math. J.},
  FJOURNAL = {Indiana University Mathematics Journal},
    VOLUME = {29},
      YEAR = {1980},
    NUMBER = {2},
     PAGES = {191--204},
      ISSN = {0022-2518},
   MRCLASS = {46L05 (06F20)},
  MRNUMBER = {563206},
MRREVIEWER = {P. A. Fillmore},
       DOI = {10.1512/iumj.1980.29.29013},
       URL = {https://doi.org/10.1512/iumj.1980.29.29013},
}

@article {WeightedMeans,
    AUTHOR = {B\"{o}ttcher, Albrecht and Garcia, Stephan Ramon and Omar, Mohamed
              and O'Neill, Christopher},
     TITLE = {Weighted means of {B}-splines, positivity of divided
              differences, and complete homogeneous symmetric polynomials},
   JOURNAL = {Linear Algebra Appl.},
  FJOURNAL = {Linear Algebra and its Applications},
    VOLUME = {608},
      YEAR = {2021},
     PAGES = {68--83},
      ISSN = {0024-3795},
   MRCLASS = {05E05 (41A15)},
  MRNUMBER = {4140644},
MRREVIEWER = {Tian-Xiao He},
       DOI = {10.1016/j.laa.2020.08.018},
       URL = {https://doi.org/10.1016/j.laa.2020.08.018},
}

@book{comtet,
    AUTHOR = {Comtet, Louis},
     TITLE = {Advanced combinatorics},
   EDITION = {enlarged},
      NOTE = {The art of finite and infinite expansions},
 PUBLISHER = {D. Reidel Publishing Co., Dordrecht},
      YEAR = {1974},
     PAGES = {xi+343},
      ISBN = {90-277-0441-4},
   MRCLASS = {05-02},
  MRNUMBER = {460128},
}

@book{schur1926additiven,
  title={Zur additiven Zahlentheorie},
  author={Schur, I.},
  series={Sitzungsberichte der Preussischen Akademie der Wissenschaften. Physikalisch-mathematische Klasse},
  url={https://books.google.com/books?id=lW39SAAACAAJ},
  year={1926}
}

@article {arnold,
    AUTHOR = {Arnold, V.~I.},
     TITLE = {Weak asymptotics of the numbers of solutions of {D}iophantine
              equations},
   JOURNAL = {Funktsional. Anal. i Prilozhen.},
  FJOURNAL = {Rossi\u\i skaya Akademiya Nauk. Funktsional$\prime$ny\u\i \ Analiz i ego
              Prilozheniya},
    VOLUME = {33},
      YEAR = {1999},
    NUMBER = {4},
     PAGES = {65--66},
      ISSN = {0374-1990},
   MRCLASS = {11P21 (11D45 11P82)},
  MRNUMBER = {1746430},
MRREVIEWER = {J. S. Joel},
       DOI = {10.1007/BF02467112},
       URL = {http://dx.doi.org/10.1007/BF02467112},
}

@article {Mundici,
    AUTHOR = {Mundici, Daniele},
     TITLE = {Farey stellar subdivisions, ultrasimplicial groups, and
              {$K_0$} of {AF} {$C^*$}-algebras},
   JOURNAL = {Adv. in Math.},
  FJOURNAL = {Advances in Mathematics},
    VOLUME = {68},
      YEAR = {1988},
    NUMBER = {1},
     PAGES = {23--39},
      ISSN = {0001-8708},
   MRCLASS = {46L80 (06F20 11H99 19K14)},
  MRNUMBER = {931170},
MRREVIEWER = {Bola O. Balogun},
       DOI = {10.1016/0001-8708(88)90006-0},
       URL = {https://doi.org/10.1016/0001-8708(88)90006-0},
}

@article {Eckhardt,
    AUTHOR = {Eckhardt, Caleb},
     TITLE = {A noncommutative {G}auss map},
   JOURNAL = {Math. Scand.},
  FJOURNAL = {Mathematica Scandinavica},
    VOLUME = {108},
      YEAR = {2011},
    NUMBER = {2},
     PAGES = {233--250},
      ISSN = {0025-5521},
   MRCLASS = {46L89},
  MRNUMBER = {2805604},
MRREVIEWER = {Marius V. Ionescu},
       DOI = {10.7146/math.scand.a-15169},
       URL = {https://doi.org/10.7146/math.scand.a-15169},
}

@article {Spielberg,
    AUTHOR = {Mitscher, Ian and Spielberg, Jack},
     TITLE = {A{F} {$C^*$}-algebras from non-{AF} groupoids},
   JOURNAL = {Trans. Amer. Math. Soc.},
  FJOURNAL = {Transactions of the American Mathematical Society},
    VOLUME = {375},
      YEAR = {2022},
    NUMBER = {10},
     PAGES = {7323--7371},
      ISSN = {0002-9947},
   MRCLASS = {46L05 (22A22 46L80)},
  MRNUMBER = {4491428},
MRREVIEWER = {Chris Bruce},
       DOI = {10.1090/tran/8723},
       URL = {https://doi.org/10.1090/tran/8723},
}

@article {Boca,
    AUTHOR = {Boca, Florin P.},
     TITLE = {An {AF} algebra associated with the {F}arey tessellation},
   JOURNAL = {Canad. J. Math.},
  FJOURNAL = {Canadian Journal of Mathematics. Journal Canadien de
              Math\'{e}matiques},
    VOLUME = {60},
      YEAR = {2008},
    NUMBER = {5},
     PAGES = {975--1000},
      ISSN = {0008-414X},
   MRCLASS = {46L05 (11A55 11B57 37E05 47L55 82B20)},
  MRNUMBER = {2442043},
MRREVIEWER = {Ryan J. Zerr},
       DOI = {10.4153/CJM-2008-043-1},
       URL = {https://doi.org/10.4153/CJM-2008-043-1},
}

@article {Bratteli72,
    AUTHOR = {Bratteli, Ola},
     TITLE = {Inductive limits of finite dimensional {$C\sp{\ast}
              $}-algebras},
   JOURNAL = {Trans. Amer. Math. Soc.},
  FJOURNAL = {Transactions of the American Mathematical Society},
    VOLUME = {171},
      YEAR = {1972},
     PAGES = {195--234},
      ISSN = {0002-9947},
   MRCLASS = {46L05},
  MRNUMBER = {312282},
MRREVIEWER = {Z. Takeda},
       DOI = {10.2307/1996380},
       URL = {https://doi.org/10.2307/1996380},
}

@article {Bratteli74,
    AUTHOR = {Bratteli, Ola},
     TITLE = {Structure spaces of approximately finite-dimensional
              {$C\sp{\ast} $}-algebras},
   JOURNAL = {J. Functional Analysis},
  FJOURNAL = {Journal of Functional Analysis},
    VOLUME = {16},
      YEAR = {1974},
     PAGES = {192--204},
      ISSN = {0022-1236},
   MRCLASS = {46L05},
  MRNUMBER = {348507},
MRREVIEWER = {J. W. Bunce},
       DOI = {10.1016/0022-1236(74)90063-9},
       URL = {https://doi.org/10.1016/0022-1236(74)90063-9},
}

@article {Bratteli78,
    AUTHOR = {Bratteli, Ola and Elliott, George A.},
     TITLE = {Structure spaces of approximately finite-dimensional
              {$C\sp{\ast} $}-algebras. {II}},
   JOURNAL = {J. Functional Analysis},
  FJOURNAL = {Journal of Functional Analysis},
    VOLUME = {30},
      YEAR = {1978},
    NUMBER = {1},
     PAGES = {74--82},
      ISSN = {0022-1236},
   MRCLASS = {46L05},
  MRNUMBER = {513479},
MRREVIEWER = {J. W. Bunce},
       DOI = {10.1016/0022-1236(78)90056-3},
       URL = {https://doi.org/10.1016/0022-1236(78)90056-3},
}

@article {Elliott76,
    AUTHOR = {Elliott, George A.},
     TITLE = {On the classification of inductive limits of sequences of
              semisimple finite-dimensional algebras},
   JOURNAL = {J. Algebra},
  FJOURNAL = {Journal of Algebra},
    VOLUME = {38},
      YEAR = {1976},
    NUMBER = {1},
     PAGES = {29--44},
      ISSN = {0021-8693},
   MRCLASS = {46L05 (16A46)},
  MRNUMBER = {397420},
MRREVIEWER = {Horst Behncke},
       DOI = {10.1016/0021-8693(76)90242-8},
       URL = {https://doi.org/10.1016/0021-8693(76)90242-8},
}

@article {CHS-Norms,
    AUTHOR = {Aguilar, Konrad and Ch\'{a}vez, \'{A}ngel and Garcia, Stephan Ramon
              and Vol\v{c}i\v{c}, Jurij},
     TITLE = {Norms on complex matrices induced by complete homogeneous
              symmetric polynomials},
   JOURNAL = {Bull. Lond. Math. Soc.},
  FJOURNAL = {Bulletin of the London Mathematical Society},
    VOLUME = {54},
      YEAR = {2022},
    NUMBER = {6},
     PAGES = {2078--2100},
      ISSN = {0024-6093},
   MRCLASS = {47A30 (15A60 16R30)},
  MRNUMBER = {4523751},
       DOI = {10.1112/blms.12679},
       URL = {https://doi.org/10.1112/blms.12679},
}

@article {RVN1,
    AUTHOR = {Ch\'{a}vez, \'{A}ngel and Garcia, Stephan Ramon and Hurley, Jackson},
     TITLE = {Norms on complex matrices induced by random vectors},
   JOURNAL = {Canad. Math. Bull.},
  FJOURNAL = {Canadian Mathematical Bulletin. Bulletin Canadien de
              Math\'{e}matiques},
    VOLUME = {66},
      YEAR = {2023},
    NUMBER = {3},
     PAGES = {808--826},
      ISSN = {0008-4395},
   MRCLASS = {60B20 (05E05 15A60)},
  MRNUMBER = {4651637},
       DOI = {10.4153/s0008439522000741},
       URL = {https://doi.org/10.4153/s0008439522000741},
}

@article {RVN2,
    AUTHOR = {Ch\'{a}vez, \'{A}ngel and Garcia, Stephan Ramon and Hurley, Jackson},
     TITLE = {Norms on complex matrices induced by random vectors {II}:
              extension of weakly unitarily invariant norms},
   JOURNAL = {Canad. Math. Bull.},
  FJOURNAL = {Canadian Mathematical Bulletin. Bulletin Canadien de
              Math\'{e}matiques},
    VOLUME = {67},
      YEAR = {2024},
    NUMBER = {2},
     PAGES = {447--457},
      ISSN = {0008-4395},
   MRCLASS = {47A30 (15A60 16R30)},
  MRNUMBER = {4751519},
       DOI = {10.4153/S0008439523000875},
       URL = {https://doi.org/10.4153/S0008439523000875},
}

@book {DavisInterpolation,
    AUTHOR = {Davis, Philip J.},
     TITLE = {Interpolation and approximation},
      NOTE = {Republication, with minor corrections, of the 1963 original,
              with a new preface and bibliography},
 PUBLISHER = {Dover Publications, Inc., New York},
      YEAR = {1975},
     PAGES = {xv+393},
   MRCLASS = {41-02 (42-02)},
  MRNUMBER = {380189},
}

@article{dixmier,
title = {On some {$C^*$}-algebras considered by {G}limm},
journal = {Journal of Functional Analysis},
volume = {1},
number = {2},
pages = {182-203},
year = {1967},
issn = {0022-1236},
doi = {https://doi.org/10.1016/0022-1236(67)90031-6},
url = {https://www.sciencedirect.com/science/article/pii/0022123667900316},
author = {J Dixmier}
}

@article{factorization2,
   AUTHOR = {Garcia, Stephan Ramon and Omar, Mohamed and O'Neill,
              Christopher and Yih, Samuel},
     TITLE = {Factorization length distribution for affine semigroups {II}:
              asymptotic behavior for numerical semigroups with arbitrarily
              many generators},
   JOURNAL = {J. Combin. Theory Ser. A},
  FJOURNAL = {Journal of Combinatorial Theory. Series A},
    VOLUME = {178},
      YEAR = {2021},
     PAGES = {Paper No. 105358, 34},
      ISSN = {0097-3165,1096-0899},
   MRCLASS = {20M14 (05E05 20M13)},
  MRNUMBER = {4175889},
MRREVIEWER = {Kurt\ D.\ Herzinger},
       DOI = {10.1016/j.jcta.2020.105358},
       URL = {https://doi.org/10.1016/j.jcta.2020.105358},
}

@misc{factorization5,
      title={Factorization length distribution for affine semigroups {V}: explicit asymptotic behavior of weighted factorization lengths on numerical semigroups}, 
      author={Stephan Ramon Garcia and Gabe Udell},
      year={2025},
      eprint={2503.01027},
      archivePrefix={arXiv},
      primaryClass={math.CO},
      note={\url{https://arxiv.org/abs/2503.01027}} 
}

@misc{davidson,
      title={{$C^*$} {Algebras} by {Example}}, 
      author={Kenneth R. Davidson},
    volume={6},
      year={1996},
isbn={978-0-8218-0599-2},
publisher={ Providence, R.I. : American Mathematical Society}
}

@book {Murphy,
    AUTHOR = {Murphy, Gerard J.},
     TITLE = {{$C^*$}-algebras and operator theory},
 PUBLISHER = {Academic Press, Inc., Boston, MA},
      YEAR = {1990},
     PAGES = {x+286},
      ISBN = {0-12-511360-9},
   MRCLASS = {46Lxx (46-01)},
  MRNUMBER = {1074574},
MRREVIEWER = {E. Gerlach},
}

@book{lectures-on-op-theory,
	series = {Fields {Institute} monographs},
	title = {Lectures on {Operator} {Theory}},
	isbn = {978-0-8218-0821-4},
	url = {https://books.google.ca/books?id=afHBx8081RIC},
	publisher = {American Mathematical Society},
	author = {Bhat, B.V.R. and Elliott, G.A. and Fillmore, P.A.},
	year = {1999},
	lccn = {99052254},
}

@book {Phillips,
    AUTHOR = {Phillips, George M.},
     TITLE = {Interpolation and approximation by polynomials},
    SERIES = {CMS Books in Mathematics/Ouvrages de Math\'{e}matiques de la SMC},
    VOLUME = {14},
 PUBLISHER = {Springer-Verlag, New York},
      YEAR = {2003},
     PAGES = {xiv+312},
      ISBN = {0-387-00215-4},
   MRCLASS = {41-01 (41A05 41A10)},
  MRNUMBER = {1975918},
MRREVIEWER = {Allan Pinkus},
       DOI = {10.1007/b97417},
       URL = {https://doi.org/10.1007/b97417},
}

@article{curry-schoenberg,
	title = {On {Pólya} frequency functions {IV}: {The} fundamental spline functions and their limits},
	volume = {17},
	issn = {1565-8538},
	url = {https://doi.org/10.1007/BF02788653},
	doi = {10.1007/BF02788653},
	number = {1},
	journal = {Journal d’Analyse Mathématique},
	author = {Curry, H. B. and Schoenberg, I. J.},
	month = dec,
	year = {1966},
	pages = {71--107},
}

@article{marechal, title={Sur La Classification Des Symetries Des {$C^*$}-Algebres {UHF}}, volume={31}, DOI={10.4153/CJM-1979-055-7}, number={3}, journal={Canadian Journal of Mathematics}, author={Fack, Th. and Marechal, O.}, year={1979}, pages={496–523}}

@book {Assi,
    AUTHOR = {Assi, Abdallah and D'Anna, Marco and Garc\'{\i}a-S\'{a}nchez, Pedro A.},
     TITLE = {Numerical semigroups and applications},
    SERIES = {RSME Springer Series},
    VOLUME = {3},
      NOTE = {Second edition [of  3558713]},
 PUBLISHER = {Springer, Cham},
      YEAR = {[2020] \copyright 2020},
     PAGES = {xiv+138},
      ISBN = {978-3-030-54942-8; 978-3-030-54943-5},
   MRCLASS = {20M14},
  MRNUMBER = {4230109},
       DOI = {10.1007/978-3-030-54943-5},
       URL = {https://doi.org/10.1007/978-3-030-54943-5},
}

@book{Rosales,
	series = {Developments in {Mathematics}},
	title = {Numerical {Semigroups}},
	isbn = {978-1-4614-2456-7},
	url = {https://books.google.ca/books?id=DVNJuAAACAAJ},
	publisher = {Springer New York},
	author = {Rosales, J.C. and García-Sánchez, P.A.},
	year = {2012},
}

@article {deBoorRecursion,
    AUTHOR = {de Boor, Carl},
     TITLE = {On calculating with {$B$}-splines},
   JOURNAL = {J. Approximation Theory},
  FJOURNAL = {Journal of Approximation Theory},
    VOLUME = {6},
      YEAR = {1972},
     PAGES = {50--62},
      ISSN = {0021-9045,1096-0430},
   MRCLASS = {41A15},
  MRNUMBER = {338617},
MRREVIEWER = {C.\ A.\ Hall},
       DOI = {10.1016/0021-9045(72)90080-9},
       URL = {https://doi.org/10.1016/0021-9045(72)90080-9},
}

@book {practicalGuide,
    AUTHOR = {de Boor, Carl},
     TITLE = {A practical guide to splines},
    SERIES = {Applied Mathematical Sciences},
    VOLUME = {27},
   EDITION = {Revised},
 PUBLISHER = {Springer-Verlag, New York},
      YEAR = {2001},
     PAGES = {xviii+346},
      ISBN = {0-387-95366-3},
   MRCLASS = {41-01 (41A15 65D05 65D07 65D10)},
  MRNUMBER = {1900298},
MRREVIEWER = {Gerlind\ Plonka},
}

@book {Handbook,
    AUTHOR = {Micula, Gheorghe and Micula, Sanda},
     TITLE = {Handbook of splines},
    SERIES = {Mathematics and its Applications},
    VOLUME = {462},
 PUBLISHER = {Kluwer Academic Publishers, Dordrecht},
      YEAR = {1999},
     PAGES = {xvi+604},
      ISBN = {0-7923-5503-2},
   MRCLASS = {41-00 (00A15 41-01 41A15 65D07)},
  MRNUMBER = {1673026},
       DOI = {10.1007/978-94-011-5338-6},
       URL = {https://doi.org/10.1007/978-94-011-5338-6},
}

@article {NathansonParts,
    AUTHOR = {Nathanson, Melvyn B.},
     TITLE = {Partitions with parts in a finite set},
   JOURNAL = {Proc. Amer. Math. Soc.},
  FJOURNAL = {Proceedings of the American Mathematical Society},
    VOLUME = {128},
      YEAR = {2000},
    NUMBER = {5},
     PAGES = {1269--1273},
      ISSN = {0002-9939},
   MRCLASS = {11P82 (05A17)},
  MRNUMBER = {1705753},
MRREVIEWER = {Yifan Yang},
       DOI = {10.1090/S0002-9939-00-05606-9},
       URL = {https://doi.org/10.1090/S0002-9939-00-05606-9},
}

@book {RamirezDiop,
    AUTHOR = {Ram\'irez Alfons\'in, J. L.},
     TITLE = {The {D}iophantine {F}robenius problem},
    SERIES = {Oxford Lecture Series in Mathematics and its Applications},
    VOLUME = {30},
 PUBLISHER = {Oxford University Press, Oxford},
      YEAR = {2005},
     PAGES = {xvi+243},
      ISBN = {978-0-19-856820-9; 0-19-856820-7},
   MRCLASS = {11D72 (05A17 11D04 11Y50 20M14)},
  MRNUMBER = {2260521},
MRREVIEWER = {Ali Sinan Sert\"oz},
       DOI = {10.1093/acprof:oso/9780198568209.001.0001},
       URL = {https://doi.org/10.1093/acprof:oso/9780198568209.001.0001},
}

@article {alievhenkaicke,
    AUTHOR = {Aliev, I. and Henk, M. and Hinrichs, A.},
     TITLE = {Expected {F}robenius numbers},
   JOURNAL = {J. Combin. Theory Ser. A},
  FJOURNAL = {Journal of Combinatorial Theory. Series A},
    VOLUME = {118},
      YEAR = {2011},
    NUMBER = {2},
     PAGES = {525--531},
      ISSN = {0097-3165},
   MRCLASS = {11D07 (11D04 11D45 11D85 11H06)},
  MRNUMBER = {2739501},
MRREVIEWER = {Clemens Fuchs},
       DOI = {10.1016/j.jcta.2009.12.012},
       URL = {http://dx.doi.org/10.1016/j.jcta.2009.12.012},
}

@article {RNS,
    AUTHOR = {De Loera, Jesus and O'Neill, Christopher and Wilburne, Dane},
     TITLE = {Random numerical semigroups and a simplicial complex of
              irreducible semigroups},
   JOURNAL = {Electron. J. Combin.},
  FJOURNAL = {Electronic Journal of Combinatorics},
    VOLUME = {25},
      YEAR = {2018},
    NUMBER = {4},
     PAGES = {Paper 4.37, 16},
   MRCLASS = {20M14 (05E45)},
  MRNUMBER = {3891104},
MRREVIEWER = {D. Llena},
}

@article {bourgainsinai,
    AUTHOR = {Bourgain, J. and Sina\u\i , Ya. G.},
     TITLE = {Limit behavior of large {F}robenius numbers},
   JOURNAL = {Uspekhi Mat. Nauk},
  FJOURNAL = {Rossi\u\i skaya Akademiya Nauk. Moskovskoe Matematicheskoe
              Obshchestvo. Uspekhi Matematicheskikh Nauk},
    VOLUME = {62},
      YEAR = {2007},
    NUMBER = {4(376)},
     PAGES = {77--90},
      ISSN = {0042-1316},
   MRCLASS = {11D85 (11D04)},
  MRNUMBER = {2358737},
MRREVIEWER = {S. A. Stepanov},
       DOI = {10.1070/RM2007v062n04ABEH004429},
       URL = {http://dx.doi.org/10.1070/RM2007v062n04ABEH004429},
}

@book {Wegge-Olsen,
    AUTHOR = {Wegge-Olsen, N. E.},
     TITLE = {{$K$}-theory and {$C^*$}-algebras},
    SERIES = {Oxford Science Publications},
      NOTE = {A friendly approach},
 PUBLISHER = {The Clarendon Press, Oxford University Press, New York},
      YEAR = {1993},
     PAGES = {xii+370},
      ISBN = {0-19-859694-4},
   MRCLASS = {46L80 (19Kxx 46L05)},
  MRNUMBER = {1222415},
MRREVIEWER = {Mahmood\ Khoshkam},
}

@article {Jacelon2025,
    AUTHOR = {Jacelon, Bhishan and Khavkine, Igor},
     TITLE = {Operator {$K$}-theoretic analysis of random adjacency
              matrices},
   JOURNAL = {New York J. Math.},
  FJOURNAL = {New York Journal of Mathematics},
    VOLUME = {31},
      YEAR = {2025},
     PAGES = {749--791},
      ISSN = {1076-9803},
   MRCLASS = {46L35 (15B52 19K99 37B10 46L80 60B20)},
  MRNUMBER = {4905722},
MRREVIEWER = {Zhichao\ Liu},
}

\appendix


\end{document}